\documentclass{article}
\usepackage[margin=4cm]{geometry}

\usepackage{amsthm,amsmath,amsfonts,amssymb,enumitem,xcolor}
\usepackage[authoryear]{natbib}
\definecolor{linkcolor}{rgb}{0, 0, 0.54}

\usepackage{graphicx}

\theoremstyle{plain}
\newtheorem{theorem}{Theorem}

\newtheorem{lemma}[theorem]{Lemma}
\newtheorem{proposition}[theorem]{Proposition}
\newtheorem{corollary}[theorem]{Corollary}

\theoremstyle{definition} 
\newtheorem{definition}[theorem]{Definition}
\newtheorem{assumption}[theorem]{Assumption}

\usepackage{titlesec}
\titlespacing*{\section}{0em}{2em}{0em}
\titlespacing*{\subsection}{0em}{2em}{0em}
\titlespacing*{\subsubsection}{0em}{2em}{0em}

\setlist[itemize]{topsep=0pt}
\setlist[enumerate]{topsep=0pt}

\usepackage{subcaption}
\usepackage{dsfont}

\usepackage[colorlinks,allcolors=linkcolor]{hyperref}

\newcommand{\DD}[1]{\ensuremath{\mathcal{D}^{{#1}}}}
\newcommand{\D}{\DD{d}}
\newcommand{\f}{\ensuremath{f^*}}
\renewcommand{\S}{\ensuremath{s}}

\newcommand{\R}{\mathbb{R}}
\newcommand{\N}{\mathbb{N}}
\newcommand{\blank}{\makebox[1ex]{\textbf{$\cdot$}}}
\newcommand\independent{\protect\mathpalette{\protect\independenT}{\perp}}
\def\independenT#1#2{\mathrel{\rlap{$#1#2$}\mkern2mu{#1#2}}}
\renewcommand{\phi}{\varphi}
\renewcommand{\epsilon}{\varepsilon}
\newcommand*\diff{\mathop{}\!\mathrm{d}}

\newcommand\smallO{\textit{o}}
\newcommand\bigO{\textit{O}}
\newcommand{\midd}{\; \middle|\;}
\newcommand{\1}{\mathds{1}}

\DeclareMathOperator*{\argmin}{\arg\!\min}

\newcommand{\arrow}[1]{\xrightarrow{\; {#1} \;}}
\newcommand{\leb}{\ensuremath{\lambda}} 
\newcommand{\KL}{\ensuremath{D_{\mathrm{KL}}}}

\newcommand{\empmeas}{{\mathbb{P}}_n} 
\DeclareMathOperator{\E}{\mathbb{E}} 

\title{Estimating conditional hazard functions and densities with
  the highly-adaptive lasso}

\usepackage{authblk}
\author[1]{Anders Munch}
\author[1]{Thomas~A.~Gerds}
\author[2]{Mark~J.~van~der~Laan}
\author[1]{Helene~C.~W.~Rytgaard}
\affil[1]{Section of Biostatistics, University of Copenhagen}
\affil[2]{Devision of Biostatistics, University of California, Berkeley}

\begin{document}

\maketitle

\begin{abstract}
  We consider estimation of conditional hazard functions and densities over the
  class of multivariate càdlàg functions with uniformly bounded sectional
  variation norm when data are either fully observed or subject to
  right-censoring. We demonstrate that the empirical risk minimizer is either
  not well-defined or not consistent for estimation of conditional hazard
  functions and densities. Under a smoothness assumption about the
  data-generating distribution, a highly-adaptive lasso estimator based on a
  particular data-adaptive sieve achieves the same convergence rate as has been
  shown to hold for the empirical risk minimizer in settings where the latter is
  well-defined. We use this result to study a highly-adaptive lasso estimator of
  a conditional hazard function based on right-censored data. We also propose a
  new conditional density estimator and derive its convergence rate. Finally, we
  show that the result is of interest also for settings where the empirical risk
  minimizer is well-defined, because the highly-adaptive lasso depends on a much
  smaller number of basis function than the empirical risk minimizer.
\end{abstract}

\section{Introduction}
\label{sec:introduction}

Let \(\D_M\) be the space of multivariate càdlàg functions
\(f\colon[0,1]^d \rightarrow \R \) with sectional variation norm bounded by
\(M<\infty\). For a suitable loss function
\(L \colon \D_M \times [0,1] \rightarrow \R\) and a probability measure \(P\) on
\([0,1]^d\), we consider the parameter
\begin{equation}
  \label{eq:est-prob-d1}
  \f = \argmin_{f \in \D_M}P{[L(f, \blank)]},
  \quad \text{where} \quad P{[L(f, \blank)]} = \int_{[0,1]^d} L(f, x)  \diff P(x).
\end{equation}
The empirical risk minimizer estimates \( \f \) by minimizing
\( \empmeas{[L(f, \blank)]}\) over \( \D_M \), where \( \empmeas \) is the
empirical measure of a dataset \(\{X_i\}_{i=1}^n\) of i.i.d.\ observations
\( X_i \sim P \). The highly-adaptive lasso (HAL) estimator proposed by
\cite{van2017generally} estimates \( \f \) by minimizing
\( \empmeas{[L(f, \blank)]}\) over a sieve, i.e., a growing subset of the
parameter space
\citep{grenander1981abstract,geman1981sieves,geman1982nonparametric,walter1984simple}.
For particular choices of loss functions and sieves, the HAL estimator and the
empirical risk minimizer will be identical, but they might also be different. We
demonstrate that the use of a sieve is necessary for conditional hazard and
density estimation, as empirical risk minimizers over the class of cadlag
functions are not formally well-defined in these settings. A recent result by
\cite{van2023higher} formally established that a particular data-adaptive choice
of sieve is sufficient to achieve the asymptotic convergence rate of
\( n^{-1/3}\log(n)^{2(d-1)/3} \) when a smoothness assumption is imposed on the
measure \( P \). We use this result to theoretically study a conditional hazard
function estimator and a novel conditional density estimator based on a HAL.

Estimation of function-valued parameters over the function class \( \D_M \) is
of interest because the bound on the sectional variation norm works like a
non-parametric sparsity constraint that to some extend allows us to avoid the
curse of dimensionality. A particularly important application is in targeted or
debiased machine learning \citep{van2011targeted,chernozhukov2018double}, for
which non-parametric estimators of regression functions, conditional densities,
and conditional hazard functions are needed. A targeted or debiased estimator
relies on the ability to estimate such nuisance parameters faster than rate
\( n^{-1/4} \). It has been shown that this rate can be achieved independently
of the dimension of the covariate space when the nuisance parameter is assumed
to belong to \( \D_M \) \citep{van2017generally,bibaut2019fast}. In this paper
we take a closer look at this important result for conditional densities and
hazard functions. In addition, our work is relevant for estimation of regression
functions: The HAL estimator can be constructed using a number of basis
functions that scales linearly in the sample size \(n\) while the empirical risk
minimizer needs a number of basis function that is of order \(n^d\)
\citep{fang2021multivariate}.

The challenge with estimation of densities and hazard functions over \( \D_M \)
is illustrated in Figure~\ref{fig:emp-risk-dens-problem}. Informally, the
sectional variation norm measures how much a function fluctuates without taking
into account where in the domain of the function the fluctuations happen. A
consequence is that we can redistribute the probability mass assigned by a given
càdlàg density function without changing its sectional variation norm in such a
way that the log-likelihood loss is decreased. A similar issue occurs with
right-censored data in survival analysis, and our
Proposition~\ref{prop:emp-risk-min-surv-dont-exist} formally shows that the
empirical risk minimizer is in general either not well-defined or not consistent
for the conditional hazard function. On the other hand, a consistent HAL
estimator of a conditional hazard function does exist.

\begin{figure}
  \hspace*{\fill}
  \begin{subfigure}[h]{0.32\linewidth}
    \caption{}
    \label{fig:dens-fig-a}
    \centering \includegraphics[width=\linewidth, trim={0.6cm 0 0 0},clip]{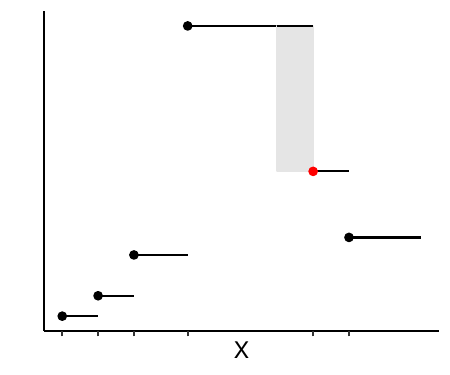}
  \end{subfigure}
  \hfill
  \begin{subfigure}[h]{0.32\linewidth}
    \caption{}
    \label{fig:dens-fig-b}
    \centering \includegraphics[width=\linewidth, trim={0.6cm 0 0
      0},clip]{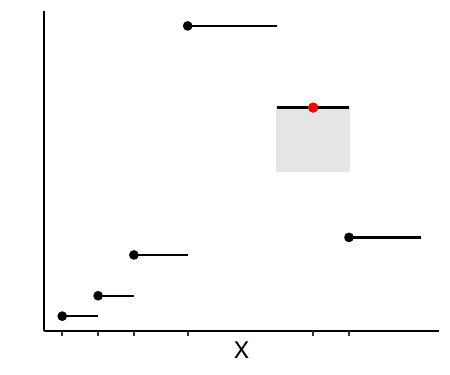}
  \end{subfigure}
  \hspace*{\fill}
  \caption[]{Illustration of two càdlàg densities, where the ticks at the
    \( x \)-axis denote observed data points, and the \( y \)-coordinates of the
    black dots denote the likelihood given to these points by the densities.
    Panel~(\subref{fig:dens-fig-a}) shows a given càdlàg density, while
    panel~(\subref{fig:dens-fig-b}) shows an adjusted density that has the same
    variation norm but assigns a higher likelihood to the observed data. Note
    that the function in panel~(\subref{fig:dens-fig-b}) is a density, because
    the gray boxes in panels~(\subref{fig:dens-fig-a})
    and~(\subref{fig:dens-fig-b}) have the same area. }
  \label{fig:emp-risk-dens-problem}
\end{figure}

Earlier related work on non-parametric functional estimation used Sobolev spaces
\citep{Gold:Mess:92:OPE,bickel88:density,Ston:80:ORC, goldstein1996efficient}.
Estimation over the class of multivariate càdlàg functions with uniformly
bounded sectional variation norm was introduced in \citep{van2017generally}.
Estimation of conditional hazard functions in the presence of censoring has
traditionally been done using kernel smoothing or local linear polynomials
\citep[e.g.,][]{ramlau1983smoothing,mckeague1990inference,van2001hazard,spierdijk2008nonparametric},
while more recent approaches use boosting
\citep{schmid2008flexible,hothorn2020transformation,lee2021boosted}. Conditional
hazard function estimation based on HAL was proposed by
\cite{rytgaard2022continuous,rytgaard2021estimation}.
\cite{fang2021multivariate} considered estimation of regression functions over
the class of functions with uniformly bounded Hardy-Krause variation
\citep{krause1903mittelwertsatze,hardy1906double,owen2005multidimensional,ChristophAistleitner2015},
which is closely related to the class of functions considered here.

The remainder of the article is organized as follows. In
Section~\ref{sec:funct-bound-sect} we introduce our notation and review the
properties of multivariate càdlàg functions with bounded sectional variation
norm. Section~\ref{sec:empir-risk-minim} contains a formal definition of the
general loss based estimation problem and two (potentially different)
estimators; the empirical risk minimizer and a HAL estimator. In
Section~\ref{sec:conv-rates-suff} we define a projection of \( \f \) onto a
data-adaptive sieve, which allows us to derive the asymptotic convergence rate
directly for the HAL estimator without assuming it to be identical to the
empirical risk minimizer. In
Sections~\ref{sec:censored-data}-\ref{sec:regression-function} we apply our
general results to special cases. In Section~\ref{sec:censored-data} we consider
the setting of censored survival data observed in continuous time, and show that
while the HAL estimator is well-defined, the empirical risk minimizer is in
general either ill-defined or inconsistent. In
Section~\ref{sec:density-estimation} we consider conditional density estimation
and propose a new estimator. Section~\ref{sec:regression-function} considers an
example from the regression setting, where the empirical risk minimizer is
well-defined, and we illustrate the dramatic reduction in the number of basis
functions needed to calculate the HAL estimator compared to the empirical risk
minimizer. Section~\ref{sec:discussion} contains a discussion of our results.
Appendices~\ref{sec:cadl-funct-meas}-\ref{sec:additional-proofs} contain proofs.

\section{Multivariate càdlàg functions with bounded sectional variation norm}
\label{sec:funct-bound-sect}

For \(d=1\) the definition of a càdlàg function is given by its name -- it is a
function that is continuous from the right with left-hand limits. When $d>1$ we
can approach a point from an infinite number of directions, and thus the
concepts `right' and `left' are not defined. In dimension \( d \in \N \) we
define càdlàg functions as follows. For any \(u \in [0,1]\) and \(a\in\{0,1\}\)
we define the interval
\begin{equation*}
  I_a(u) = 
  \begin{cases}
    [u, 1] & \text{if } a=1,\\
    [0, u) & \text{if } a=0.
  \end{cases}
\end{equation*}
For any \(\mathbf{u} =(u_1, \dots, u_d) \in [0,1]^d\) and
\(\mathbf{a} = (a_1, \dots, a_d) \in \{0,1\}^d\) define the quadrant
\(Q_{\mathbf{a}}(\mathbf{u}) = I_{a_1}(u_1) \times \cdots \times I_{a_d}(u_d)\).

\begin{definition}[Multivariate càdlàg function]
  \label{def:cadlag}
  A function \(f \colon [0,1]^d \rightarrow \R\) is \textit{càdlàg} if for all
  \(\mathbf{u} \in [0,1]^d\), \(\mathbf{a} \in\{0,1\}^d\), and any sequence
  \(\{\mathbf{u}_n\} \subset Q_{\mathbf{a}}(\mathbf{u}) \) which converges to
  \( \mathbf{u} \) as \(n\to\infty\), the limit
  \(\lim_{n\rightarrow\infty}f(\mathbf{u}_n)\) exists, and
  \(\lim_{n\rightarrow\infty}f(\mathbf{u}_n) = f(\mathbf{u})\) for
  \(\mathbf{a} = \mathbf{1}\).
\end{definition}

\cite{neuhaus1971weak} first generalized the concept of a càdlàg function to the
multivariate setting. Our Definition~\ref{def:cadlag} is an equivalent
definition used by, e.g., \cite{czerebak2008almost} and \cite{ferger2015arginf}.
We use \( \D \) to denote the collection of all càdlàg functions with domain
\( [0,1]^d \). The content of Defintion~\ref{def:cadlag} is illustrated in
Figure~\ref{fig:cadlag-def} for \( d=2 \).

\begin{figure}
  \centerline{\includegraphics[width=1\linewidth,trim={0.6cm 0 0 0},clip]{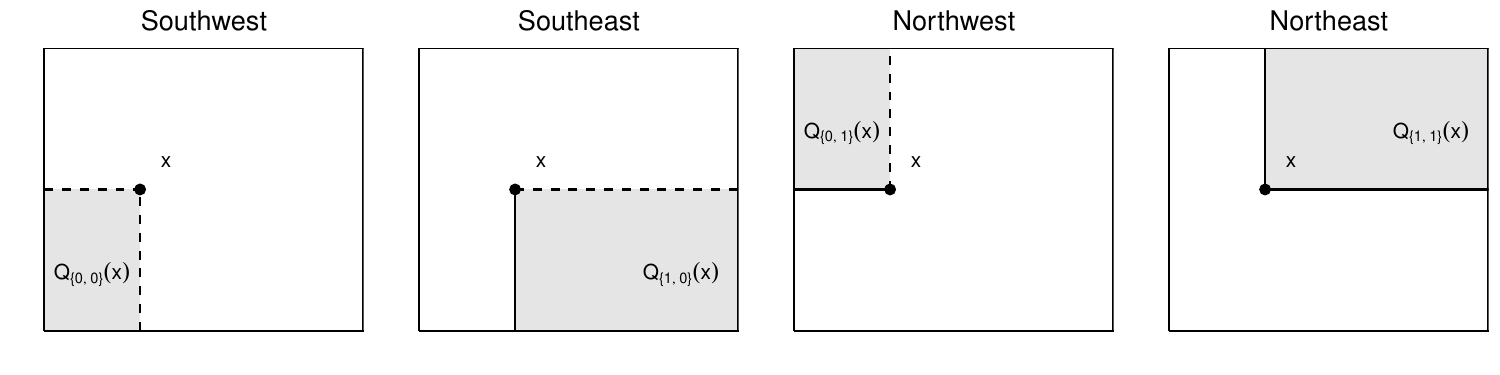}}
  \caption[]{The four quadrants \( Q_{\{0,0\}}(\mathbf{x})\),
    \( Q_{\{1,0\}}(\mathbf{x}) \), \( Q_{\{0,1\}}(\mathbf{x}) \), and
    \( Q_{\{1,1\}}(\mathbf{x}) \) spanned by the point
    \(\mathbf{x} \in [0,1]^2\) and each of the vertices in the unit square. A
    sequence which is contained in one of the quadrants and converges to
    \( u \), converges from `southwest', `southeast', `northwest', or
    `northeast'. That the function \( f \) is càdlàg means that the limit of
    the function \(f\) should exist when we approach it from any of these four
    directions, and the limit should agree with the function value at \(u\) when
    we approach it from `northeast'.}
  \label{fig:cadlag-def}
\end{figure}

A bit of notation is needed to formally define the section of a càdlàg function
and the sectional variation norm. For any non-empty subset
\(\S \subset [d] = \{1, \dots, d\}\) let
$\pi_{\S}\colon\{1, \dots, |\S|\} \rightarrow [d]$ be the unique increasing
function such that \(\mathrm{Im}{(\pi_{\S})}=\S\), i.e., $\pi_{\S}$ provides the
ordered indices of \(1, \dots, d \) included in \(\S\). For any
\( \mathbf{x} \in [0,1]^d\) we define the \( \S \)-section of the vector
\( \mathbf{x} \) as
\begin{equation*}
  \mathbf{x}_{\S} = (x_{\pi_{\S}(1)}, \dots, x_{\pi_{\S}(|{\S}|)}) \in
  [0,1]^{|{\S}|},
\end{equation*}
i.e., $\mathbf{x}_{\S}$ is the ordered tuple in \( [0,1]^{|{\S}|}\) consisting
of all components of \( \mathbf{x} \) with index in \({\S}\). Note that for a
singleton \({\S}= \{i\}\), we have \(\mathbf{x}_{\{i\}}=x_i\). Defining
\begin{equation*}
  \overline{\mathbf{x}}_s = (\1{\{1 \in \S\}}x_1, \1{\{2 \in \S\}}x_2, \dots,
  \1{\{d \in \S\}}x_d) \in [0,1]^d.
\end{equation*}
the \( \S \)-section of \( f \) is the function
\begin{equation*}
  f_{\S} \colon [0,1]^{|\S|} \longrightarrow \R  
  \quad \text{such that} \quad
  f_{\S}(\mathbf{x}_{\S}) = f(\overline{\mathbf{x}}_{\S}),
  \quad\forall \mathbf{x} \in [0,1]^d.
\end{equation*}
In words, \( f_{\S} \) is the function that appears when all arguments of
\( f \) that are not in \( \S \) are fixed at zero. For vectors
\(\mathbf{a}, \mathbf{b} \in [0,1]^d\) we write
\begin{align*}
  & \mathbf{a} \preceq \mathbf{b} \quad \text{if} \quad a_k \leq b_k,
    \quad \text{for} \quad k = 1, \dots, d,
  \\
  & \mathbf{a} \prec \mathbf{b} \quad \text{if} \quad a_k < b_k,
    \quad \text{for} \quad k = 1, \dots, d,
\end{align*}
and we define closed and half-open boxes by
\([\mathbf{a}, \mathbf{b}] = \{ \mathbf{x} \in [0,1]^d : \mathbf{a} \preceq
\mathbf{x} \preceq \mathbf{b} \}\) and
\((\mathbf{a}, \mathbf{b}] = \{ \mathbf{x} \in [0,1]^d : \mathbf{a} \prec
\mathbf{x} \preceq \mathbf{b} \}\), respectively. For a box
\( A=(\mathbf{a}, \mathbf{b}] \subset [0,1]^d \) with
\( \mathbf{a} \prec \mathbf{b} \), let
\begin{equation*}
  \mathcal{V}(A) = \{\mathbf{v} = (v_1, \dots, v_d) : v_i = a_i \text{ or } v_i
  = b_i\}
\end{equation*}
denote the set of vertices of the box \(A\). The \textit{quasi-volume} assigned
to the box \(A=(\mathbf{a}, \mathbf{b}]\) by the function \(f\) is
\begin{equation*}
  \Delta(f ; A) = \sum_{\mathbf{v} \in \mathcal{V}(A)}
  (-1)^{H(\mathbf{v})}f(\mathbf{v}),
  \quad \text{with} \quad
  H(\mathbf{v}) = \sum_{k=1}^{d}\1{\{v_k=a_k\}}.
\end{equation*}
The idea is that the volume of the box \( A \) as measured by \( f \) can be
computed by calculating volumes of boxes with corners at \( \mathbf{0} \), as
illustrated in Figure~\ref{fig:vitali} for \( d=2 \).
\begin{figure}
  \centerline{\includegraphics[width=1\linewidth,trim={0.6cm 0 0 0},clip]{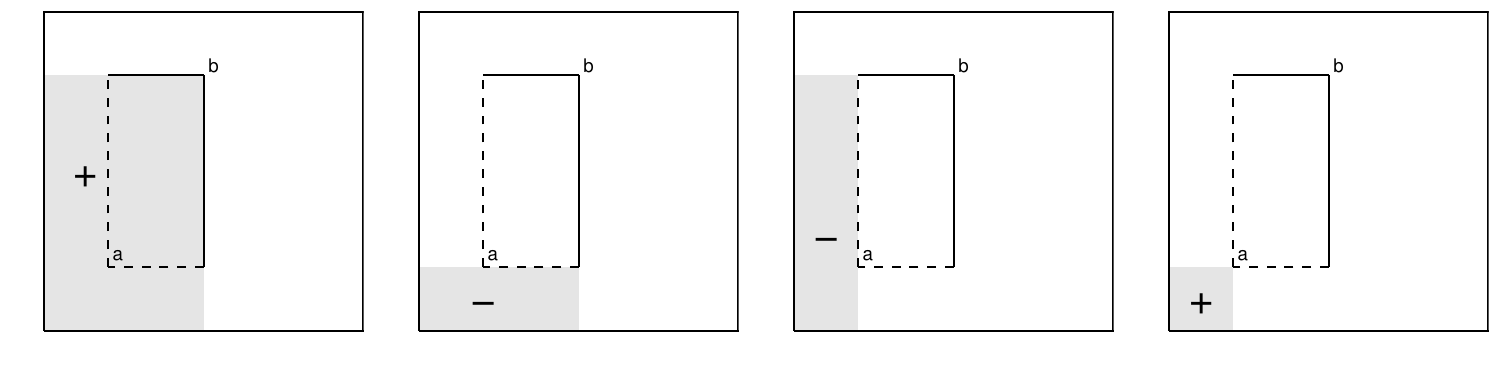}}
  \caption[]{The area of the box \( (\mathbf{a}, \mathbf{b}] \) can be
    calculated by first calculating the gray area in the leftmost figure,
    subtracting the gray areas in the two middle figures, and then adding the
    gray area in the rightmost figure.}
  \label{fig:vitali}
\end{figure}
Let \(\rho\) denote a finite partition of \((0,1]\) given by
\begin{equation*}
  \rho = \{ (x_{l-1}, x_{l}] : l=1, \dots, L\},
  \quad \text{with} \quad 0=x_0 < x_1 < \cdots < x_L =1.
\end{equation*}
For any collection \(\rho_1, \dots, \rho_d\) of univariate partitions, we define
a partition \(\mathcal{P}\) of \((0, 1]^d\) by
\begin{equation}
  \label{eq:partition}
  \mathcal{P} = \{ I_1 \times I_2 \times \cdots \times I_d : I_k \in \rho_k,
  k=1, \dots, d\}.
\end{equation}
For a function \(f \colon [0,1]^d \rightarrow \R\) the \textit{Vitali variation}
is defined as
\begin{equation*}
  V(f) = \sup_{\mathcal{P}} \sum_{A\in \mathcal{P}} |\Delta(f;A)|,
\end{equation*}
where the supremum is taken over all partitions given by
equation~(\ref{eq:partition}). The \textit{sectional variation norm} of a
function \( f\colon [0,1]^d\rightarrow\R \) is the sum of the Vitali variation
of all its sections plus the absolute value of the function at \( \mathbf{0} \),
i.e.,
\begin{equation*} 
  \| f \|_v = |f(\mathbf{0})| +
  \sum_{\S \in \mathcal{S}} 
  V(f_{\S}),
  \quad \text{with} \quad
  \mathcal{S} = \{ s \subset [d] : s \not = \emptyset\}.
\end{equation*}
For \(M\in(0,\infty)\) we use \( \D_M \) to denote the space of càdlàg functions
\(f\colon [0,1]^d \rightarrow \R\) with \(\|f\|_v \leq M \).

We now give two alternative descriptions of \( \D_M \). The first characterizes
\( \D_M \) as the closure of the collection of rectangular piece-wise constant
functions (Proposition~\ref{prop:discrete-measure-closure}), and the second puts
\( \D_M \) into a one-to-one correspondence with finite signed measures
(Proposition~\ref{prop:cadlag-measure}).

Define the function space
\( \mathcal{F}^d = \{ \1_{[\mathbf{x}, \mathbf{1}]} : \mathbf{x} \in [0,1]^d\}
\), let \( \mathrm{Span}{(\mathcal{F}^d)} \) denote all linear combinations of
elements from \( \mathcal{F}^d \), and define
\( \mathcal{R}^d_M = \{f \in \mathrm{Span}{(\mathcal{F}^d)} : \|f\|_v \leq M \}
\). An example of an element in \( \mathcal{F}^d \) is shown in
Figure~\ref{fig:piece-wise}~(\subref{fig:piece-wise1}).
\begin{proposition}
  \label{prop:discrete-measure-closure}
  Consider \( \mathcal{R}_M^d \) and \( \D_M \) as subspaces of the Banach space
  of all bounded functions \( f \colon [0,1]^d \rightarrow \R \) equipped with
  the supremum norm. Then \( \D_M = \overline{\mathcal{R}_M^d} \), that is,
  \( \mathcal{R}_M^d \subset \D_M \) and for any function \( f \in \D_M \) there
  exists a sequence of functions \( \{f_n\} \subset \mathcal{R}_M^d \) such that
  \( \Vert f- f_n\Vert_{\infty} \rightarrow 0\).
\end{proposition}
\begin{proof}
  See Appendix~\ref{sec:cadl-funct-meas}.
\end{proof}
In the following, let
\( \| \mu \|_{\mathrm{TV}} = \mu_+([0,1]^d) + \mu_-([0,1]^d) \) denote the total
variation norm of the measure \( \mu \).
\begin{proposition}
  \label{prop:cadlag-measure}
  For any \( f \in \D_M \) there exists a unique signed measure \( \mu_f \) on
  \( [0,1]^d \) such that
  \begin{equation*}
    f(\mathbf{x}) = \mu_f([\mathbf{0},\mathbf{x}]), \quad \forall \mathbf{x} \in [0,1]^d,
  \end{equation*}
  and \( \| \mu_f \|_{\mathrm{TV}} = \| f \|_v \). For any signed measure
  \( \mu \) on \( [0,1]^d \) with \( \| \mu_f \|_{\mathrm{TV}} \leq M \) there
  exists a unique function \( f_{\mu} \in \D_M \) such that
  \begin{equation*}
    f_{\mu}(\mathbf{x}) = \mu([\mathbf{0},\mathbf{x}]), \quad \forall \mathbf{x} \in [0,1]^d.
  \end{equation*}
\end{proposition}
\begin{proof}
  A similar result is proved by \cite{ChristophAistleitner2015}. We use their
  result in our proof in Appendix~\ref{sec:cadl-funct-meas} which is for càdlàg
  functions.
\end{proof}
Proposition~\ref{prop:cadlag-measure} shows that the class of
functions considered by \cite{fang2021multivariate} is identical to
the class \( \D_M \) up to a constant.

Based on Proposition~\ref{prop:cadlag-measure} we can define the integral with
respect to a function \( f \in \D_M\) as the integral with respect to the
measure \( \mu_f \). We use the notation \( \diff f = \diff \mu_f\) and
\( |\diff f| = \diff| \mu_f|\), where \( |\mu| = \mu_+ + \mu_- \). The
connection between càdlàg functions and measures is the key component underlying
the HAL estimator. The HAL estimator is motivated by the following
representation of functions in \( \D_M \) which is due to
\cite{gill1995inefficient} and \cite{van2017generally}.
\begin{proposition}
  \label{prop:cadlag-repr}
  For any \( f \in \D_M \) we can write
  \begin{equation*}
    f(\mathbf{x}) = f(\mathbf{0})
    + \sum_{\S  \in \mathcal{S}} \int_{[0,1]^{|s|}} \1_{(\mathbf{0}_{\S}, \mathbf{x}_{\S}]}(\mathbf{u}) \diff f_{\S}(\mathbf{u}),
  \end{equation*}
  and
  \begin{equation*}
    \Vert f \Vert_{v}
    =
    |f(\mathbf{0})| +
    \sum_{\S \in \mathcal{S}} \int_{[0,1]^{|s|}} \1_{(\mathbf{0}_{\S}, \mathbf{1}_{\S}]}(\mathbf{u}) |\diff f_{\S}|(\mathbf{u}),
\end{equation*}
 where \(  \mathcal{S} = \{ s \subset [d] : s \not = \emptyset\} \).
\end{proposition}
\begin{proof}
  See Appendix~\ref{sec:cadl-funct-meas}.
\end{proof}

Proposition~\ref{prop:discrete-measure-closure} showed that \( \D_M \) is the
closure of the piece-wise constant functions \( \mathcal{R}_M^d \).
Proposition~\ref{prop:pw-cons-cadlag} implies that piece-wise constant functions
that are not in \( \mathcal{R}^d_M \), like the ones in
Figures~\ref{fig:piece-wise}~(\subref{fig:piece-wise2})~and~(\subref{fig:piece-wise3}),
are not càdlàg.

\begin{proposition}
  \label{prop:pw-cons-cadlag}
  Let \( f : [0,1]^d \rightarrow \mathcal{K} \subset \R \) for some finite set
  \( \mathcal{K} \). If \( f \not \in \mathcal{R}_M^d \) then \( f \) is not càdlàg.
\end{proposition}
\begin{proof}
  See Appendix~\ref{sec:cadl-funct-meas}.
\end{proof}

\begin{figure}
  \hspace*{\fill}
  \begin{subfigure}[h]{0.26\linewidth}
    \caption{}
    \label{fig:piece-wise1}
    \centering \includegraphics[width=\linewidth]{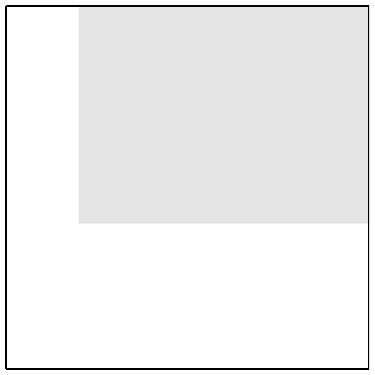}
  \end{subfigure}
  \hfill
  \begin{subfigure}[h]{0.26\linewidth}
    \caption{}
    \label{fig:piece-wise2}
    \centering \includegraphics[width=\linewidth]{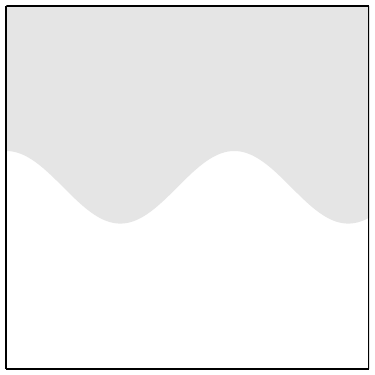}
  \end{subfigure}
  \hfill
  \begin{subfigure}[h]{0.26\linewidth}
    \caption{}
    \label{fig:piece-wise3}
    \centering \includegraphics[width=\linewidth]{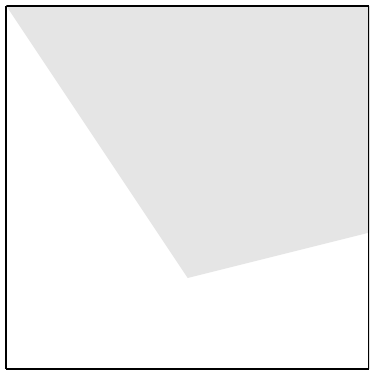}
  \end{subfigure}
  \hspace*{\fill}
  \caption[]{Illustration of functions \( f \colon [0,1]^2 \rightarrow \R \)
    that are 0 on the white area and 1 on the shaded area. The function in
    panel~(\subref{fig:piece-wise1}) is càdlàg. The functions in
    panels~(\subref{fig:piece-wise2}) and~(\subref{fig:piece-wise3}) are not
    càdlàg.}
  \label{fig:piece-wise}
\end{figure}

\section{Empirical risk minimization and the HAL estimator}
\label{sec:empir-risk-minim}

We now consider a general setup for loss-based estimation. We assume given an
i.i.d.\ dataset \( O_i \sim P \), \( i=1, \dots, n \), with data on the form
\begin{equation}
  \label{eq:27}
  O= (X, Y) \in \mathcal{O} = [0,1]^d \times \mathcal{Y},
  \quad \text{for} \quad
  \mathcal{Y} \subset \R.
\end{equation}
We use \( \empmeas \) to denote the empirical measure corresponding to a data
set \( \{O_i\}_{i=1}^n \). Let \( L \) be a loss function
\(L \colon \D_M \times \mathcal{O} \rightarrow \R\). We define the target
parameter
\begin{equation}
  \label{eq:true-min}
  \f = \argmin_{f \in\D_M} P{[L(f, \blank)]},
\end{equation}
which formally depends on \( M \) but we suppress that in the notation. A
natural estimator of \( \f \) is the substitution estimator, also known as the
\textit{the empirical risk minimizer},
\begin{equation}
  \label{eq:emp-full-min}
  \argmin_{f \in\D_M} \empmeas{[L(f, \blank)]}.
\end{equation}
The optimization problem in equation~(\ref{eq:emp-full-min}) reduces to a finite
but high-dimensional optimization problem for the squared error loss
\citep{fang2021multivariate}. We conjecture that this can be generalized to loss
functions for which we can write \citep[Assumption~2]{bibaut2019fast}
\begin{equation}
  \label{eq:point-loss}
  L(f, (\mathbf{x},y)) = \tilde{L}(f(\mathbf{x}), y),\quad \forall {f \in \D_M},
  \quad \text{for some function} \quad
  \tilde{L} \colon [0,1]^d \times \mathcal{Y} \longrightarrow \R_+.
\end{equation}
Note that equation~(\ref{eq:point-loss}) does not hold in general for the
negative log-likelihood as we demonstrate in Section~\ref{sec:censored-data}.

We now turn to define the HAL estimator \citep{van2017generally}. The HAL
estimator is motivated from the representation given by
Proposition~\ref{prop:cadlag-repr}, which shows that we can estimate
\( f \in \D_M \) by estimating the signed measures generated by its sections.
Let \( \delta_{X_{s,i}} \) be the Dirac measure at the \( \S \)-section of
\( X_i \) and define the estimator of the signed measure of the \( \S \)-section
of \( f \),
\begin{align*}
  \diff f_{\beta^{\S},n}
  & =
  \sum_{i=1}^{n}\beta_{i}^{\S}\delta_{X_{s,i}},
  \quad \text{with unknown parameter vector} \quad
  \beta^{\S} = (\beta_{1}^{\S}, \dots, \beta_{n}^{\S}) \in \R^n.
\end{align*}
This gives the following data-dependent model for estimation of \(f \in \D_M \),
\begin{equation}
  \label{eq:9}
  f_{\beta,n}(\mathbf{x})
  =
  \beta_0 +
  \sum_{\S \in \mathcal{S}}
  \sum_{i=1}^{n}\beta_{i}^{\S} \1\{X_{\S,i} \preceq \mathbf{x}_{\S}\},
  \quad \text{with} \quad
  \beta = \{\beta^{\S} : \S \in \mathcal{S} \} \cup \{\beta_0\},
\end{equation}
where \( \mathcal{S} = \{ s \subset [d] : s \not = \emptyset\} \). As
\( | \mathcal{S} | = \sum_{j=1}^{d} {\binom{d}{j}} = 2^d-1\) we have that
\(\beta \in \R^{m({d,n})}\) with \(m({d,n})= n(2^d-1)+1\). We refer to the
indicator functions in equation~(\ref{eq:9}) as basis function. Some examples of
basis functions are given in Figure~\ref{fig:hal-basis} for \( d=2 \). By
Proposition~\ref{prop:discrete-measure-closure}, any \(f_{\beta,n}\) is an
element of \(\D \) and we have
\begin{equation}
  \label{eq:8}
  \Vert f_{\beta,n} \Vert_{v}
  =
  \Vert \beta \Vert_{1}
  =
  |\beta_0| +
  \sum_{s \in \mathcal{S}}
  \sum_{i=1}^{n}|\beta_{i,s}|.
\end{equation}
Denote the space of all functions of this form by
\begin{equation*}
  \D_n :=\{f_{\beta,n} : \beta \in \R^{m(d,n)} \} \subset \D,
\end{equation*}
and denote similarly the subspace of these function with a sectional variation
norm bounded by a fixed constant \(M < \infty\) by
\begin{equation*}
  \D_{M,n} :=\{f_{\beta,n} : \beta\in \R^{m(d,n)}, \Vert \beta \Vert_{1}\leq M \} \subset \D_M.
\end{equation*}
A \textit{highly-adaptive lasso (HAL) estimator} is then defined as
\begin{equation}
  \label{eq:estimator}
  \hat{f}_n \in \argmin_{f \in\D_{M,n}} \empmeas{[L(f, \blank)]}.
\end{equation}
We refer to any minimizer as a HAL estimator.

\begin{figure}
  \captionsetup[subfigure]{labelformat=empty}
  \hspace*{\fill}
  \begin{subfigure}[h]{0.26\linewidth}
    \caption{\( \mathbf{x} \mapsto \1{\{X_{\{1\}} \preceq \mathbf{x}_{\{1\}}\}} \)}
    \label{fig:hal-basis-1}
    \centering \includegraphics[width=\linewidth]{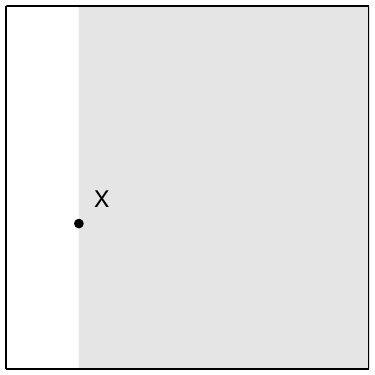}
  \end{subfigure}
  \hfill
  \begin{subfigure}[h]{0.26\linewidth}
    \caption{\( \mathbf{x} \mapsto \1{\{X_{\{2\}} \preceq \mathbf{x}_{\{2\}}\}} \)}
    \label{fig:hal-basis-2}
    \centering \includegraphics[width=\linewidth]{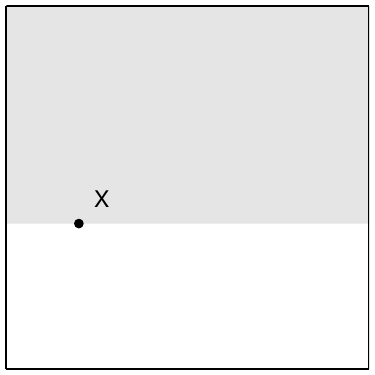}
  \end{subfigure}
  \hfill
  \begin{subfigure}[h]{0.26\linewidth}
    \caption{\( \mathbf{x} \mapsto \1{\{X_{\{1,2\}} \preceq \mathbf{x}_{\{1,2\}}\}} \)}
    \label{fig:hal-basis-12}
    \centering \includegraphics[width=\linewidth]{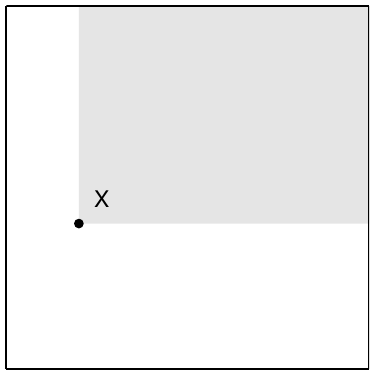}
  \end{subfigure}
  \hspace*{\fill}
  \caption[]{Examples of the basis functions that are used to construct the HAL
    estimator for \( d=2 \).}
  \label{fig:hal-basis}
\end{figure}

\section{Convergence rates using a projection}
\label{sec:conv-rates-suff}

In this section we show that a HAL estimator \(\hat{f}_n\) enjoys the same
convergence rate as has been shown to hold for the empirical risk minimizer,
when this is well-defined, under an additional smoothness assumption (see
Assumption~\ref{assum:f0} and the following discussion). In addition, we derive
asymptotic convergence rates for a HAL estimator in a setting where the
empirical risk minimizer is not well-defined. We denote by
\(N_{[\,]}(\epsilon,\mathcal{H}, \| \blank \|)\) the bracketing number for a
function space \(\mathcal{H}\) with respect to a norm \( \|\blank \| \). The
bracketing number is the minimum number of brackets with a norm smaller than
\(\epsilon\) needed to cover $\mathcal{H}$ \citep{van1996weak}. We use
\( \Vert \blank \Vert_{\infty} \) to denote the supremum norm and
\( \Vert \blank \Vert_{v} \) to denote the sectional variation norm, while for a
measure \( \mu \) we use \( \| \blank \|_{\mu} \) to denote the
\( \mathcal{L}^2(\mu) \)-norm. We use \( \leb \) to denote Lebesgue measure.
Recall that the data is of the form \( O= (X,Y) \) with \( X \in [0,1]^d \). For
all non-empty subsets \(s \subset \{1, \dots, d\}\) we let \(P_{s}\) denote the
marginal distribution of \(X_{s}\). We let \( \mu_{\f_{\S}} \) denote the
measures generated by the sections \( \f_{\S} \). Note that the measures
\( \mu_{\f_{\S}} \) and \(P_{s}\) operate on the same measure space
\( [0,1]^{|s|} \). We assume that
\( \1_{(\mathbf{0}_s, \mathbf{1}_s]} \cdot \mu_{\f_{\S}} \ll \1_{(\mathbf{0}_s,
  \mathbf{1}_s]} \cdot P_{s} \) and write the Radon-Nikodym derivatives as
\begin{equation*}
  \1_{(\mathbf{0}_s,
    \mathbf{1}_s]} \frac{ \diff \f_{\S}}{ \diff P_{s}}
  =  \frac{ \diff  \{\1_{(\mathbf{0}_s, \mathbf{1}_s]} \cdot\mu_{\f_{\S}}\}}{
    \diff  \{\1_{(\mathbf{0}_s, \mathbf{1}_s]} \cdot P_{s}\}},
  \quad \text{for} \quad
   \S \in \mathcal{S}.
\end{equation*}

\begin{assumption}[Smoothness of the loss function]
  \label{assum:loss-assum}
  For a loss function \( L \) define the function space
  \( \mathcal{L}_{M} = \left\{ L(f, \blank) : f \in \D_M \right\}\). There exist
  constants \(C < \infty\), $\eta > 0$, and $\kappa\in\N$ such that the
  following conditions hold.
  \begin{enumerate}[label=(\roman*)]
  \item \label{item:1} \(\Vert L(f, \blank) \Vert_{\infty} \leq C\) for all
    \(f\in\D_M\).
  \item \label{item:2}
    \(C^{-1}\Vert f - \f \Vert_{\leb}^2 \leq P[L(f, \blank) - L(\f, \blank)]
    \leq C \Vert f - \f \Vert_{\leb}^2\) for all \(f\in\D_M\).
  \item \label{item:5}
    \(N_{[\,]}(\epsilon, \mathcal{L}_M, \| \blank \|_{P}) \leq C
    N_{[\,]}(\epsilon/C, \D_M, \| \blank \|_{\leb} )^{\kappa}\) for all
    $\epsilon\in(0,\eta)$.
  \end{enumerate}
\end{assumption}

Assumption~\ref{assum:loss-assum}~\ref{item:2} is a standard assumption
\citep[e.g.,][]{van1996weak}. Some general conditions on the loss functions can
be given to ensure that Assumption~\ref{assum:loss-assum}~\ref{item:5} holds,
see for instance Lemma~4 in Appendix~B in \citep{bibaut2019fast}.

\begin{assumption}[Data-generating distribution]
  \label{assum:f0}
  There is a constant \( C < \infty \) such that the following conditions
  hold.
  \begin{enumerate}[label=(\roman*)]
  \item \label{item:11} The target parameter \( \f \) is an inner point of
    \( \D_M \) with respect to the sectional variation norm, i.e.,
    \( \Vert \f \Vert_{v} < M\).
  \item \label{item:3}
    \( \1_{(\mathbf{0}_s, \mathbf{1}_s]} \cdot \mu_{\f_{\S}} \ll
    \1_{(\mathbf{0}_s, \mathbf{1}_s]} \cdot P_{s} \) and
    \( \Vert \1_{(\mathbf{0}_s, \mathbf{1}_s]} \diff \f_{\S}/ \diff P_s
    \Vert_{\infty} \leq C\) for all \(s \in \mathcal{S}\).
  \end{enumerate}
\end{assumption}
Assumption~\ref{assum:f0}~\ref{item:3} is substantial, as it imposes an
additional smoothness condition on \( \f \). For instance, if \( P \) is
dominated by Lebesgue measure, Assumption~\ref{assum:f0}~\ref{item:3} implies
that the measures generated by the sections of \( \f \) must also be dominated
by Lebesgue measure, hence \( \f \) must be continuous. We discuss the necessity
of this assumption further in Section~\ref{sec:regression-function}.

Our main result (Theorem~\ref{theorem:main-result}) relies on
Lemma~\ref{lemma:projection} which is based on a construction given in
Appendix~B of \citep{van2023higher}. A proof of the lemma is given at the end of
this section.
\begin{lemma}
  \label{lemma:projection}
  For any \( f \in \mathcal{L}^2(\leb) \) and \( n \in \N \) there exists a
  (random) function \( \hat{\pi}_n(f) \in \D_{M,n} \) such that
  \begin{equation*}
    \hat{\pi}_n(f) = \argmin_{h \in \D_{M,n}}\Vert h - f \Vert_{\leb}.
  \end{equation*}
  For any \( \f \) fulfilling Assumption~\ref{assum:f0}, it holds that
  \begin{equation*}
    \Vert \hat{\pi}_n(\f) - \f \Vert_{\leb} = \bigO_P{(n^{-1/2})}.
  \end{equation*}
\end{lemma}

\begin{theorem}
  \label{theorem:main-result}
  If Assumptions~\ref{assum:loss-assum} and~\ref{assum:f0} hold, and
  \(\hat{f}_n\) is a HAL estimator as defined in equation~(\ref{eq:estimator}),
  then
  \begin{equation*}
    \Vert \hat{f}_n - \f \Vert_{\leb} = \bigO_P(n^{-1/3}\log(n)^{2(d-1)/3}).
  \end{equation*}
\end{theorem}

\begin{proof}
  We can write
  \begin{equation*}
    \Vert \hat{f}_n - \f \Vert_{\leb} \leq
    \Vert \hat{f}_n - \hat{\pi}_n(\f) \Vert_{\leb} +
    \Vert \hat{\pi}_n(\f) - \f \Vert_{\leb},
  \end{equation*}
  where \( \hat{\pi}_n \) is the projection defined in
  Lemma~\ref{lemma:projection}. As Assumption~\ref{assum:f0} is assumed to hold,
  the second term on the right hand side is of order \( \bigO_P{(n^{-1/2})} \).
  The first term of the right hand side can be analyzed using classical results
  from empirical process theory. A detailed proof showing that
  \( \Vert \hat{f}_n - \hat{\pi}_n(\f) \Vert_{\leb} =
  \bigO_P(n^{-1/3}\log(n)^{2(d-1)/3}) \) is given in
  Appendix~\ref{sec:results-from-empir}.
\end{proof}

\begin{proof}[Proof of Lemma~\ref{lemma:projection}]
  By definition of \( \D_{M,n} \), any element \( h \in \D_{M,n} \) can be
  written as \( h = \sum_{k = 1}^{m(d,n)} \beta_k h_k \), for some coefficients
  \( \beta = (\beta_1, \dots, \beta_{m(d,n)}) \) and indicator functions
  \( h_k \). Minimizing \( h \mapsto \Vert h - f \Vert_{\leb} \) over
  \( \D_{M,n} \) is thus equivalent to minimizing
  \begin{equation*}
    \mathcal{G}(\beta) = \int_{[0,1]^d}
    \left\{
      \left(
        \sum_{k=1}^{m(d,n)} \beta_k h_k
      \right)^2
      - 2\sum_{k=1}^{m(d,n)} \beta_k h_k \f
    \right\}
    \diff  \leb
  \end{equation*}
  over the set
  \(\mathcal{B}_M = \{\beta \in \R^{m(d,n)} : \Vert \beta \Vert_1 \leq M\}\).
  Writing
  \begin{equation*}
    \mathcal{G}(\beta)
    = \sum_{k=1}^{m(d,n)}\sum_{l=1}^{m(d,n)}\beta_k \beta_l \int_{[0,1]^d} h_k
    h_l  \diff \leb
    -
    2\sum_{k=1}^{m(d,n)} \beta_k
    \int_{[0,1]^d} h_k \f
    \diff  \leb,
  \end{equation*}
  shows that $\mathcal{G}$ is continuous. As \( \mathcal{B}_M \) is compact, it
  follows that a minimum is attained.

  To show the second statement of the lemma, we follow the proof of Lemma~23 in
  \citep{van2023higher} and define the random function
  \begin{equation}
    \label{eq:f-tilde-sieve}
    \f_{n}(\mathbf{x}) =\f(\mathbf{0})+
    \sum_{\S \in \mathcal{S}}\int_{(\mathbf{0}_s, \mathbf{x}_s]}\frac{\diff \f_{\S}}{\diff P_s}  \diff
    \mathbb{P}_{s,n},
  \end{equation}
  where \(\mathbb{P}_{s,n}\) is the empirical measure of the \(s\)-section of
  the data \(\{X_i\}_{i=1}^n\), i.e., the empirical measure obtained from
  \(\{X_{s, i}\}_{i=1}^n\). This function is well-defined by
  Assumption~\ref{assum:f0}~\ref{item:3}. We next show that
  \begin{equation}
    \label{eq:tilde-key-prop}
    \Vert  \f_{n} - \f \Vert_{\leb} = \bigO_P{(n^{-1/2})},
  \end{equation}
  and
  \begin{equation}    
    \label{eq:tilde-key-prop2}
        P{\left(
        \f_{n} \in \D_{M,n}
      \right)} \longrightarrow 1.
  \end{equation}
  To see this, we use the representation given by
  Proposition~\ref{prop:cadlag-repr} and Assumption~\ref{assum:f0}~\ref{item:3}
  to write
  \begin{equation*}
    \f(\mathbf{x}) =
    \f(\mathbf{0})+
    \sum_{\S \in \mathcal{S}}\int_{(\mathbf{0}_s, \mathbf{x}_s]}\diff \f_{\S}
    = \f(\mathbf{0})+
    \sum_{\S \in \mathcal{S}}\int_{(\mathbf{0}_s, \mathbf{x}_s]}\frac{\diff \f_{\S}}{\diff P_s} \diff P_s,
  \end{equation*}
  from which we obtain
  \begin{equation*}
    \f_{n}(\mathbf{x}) - \f(\mathbf{x})
    = \sum_{\S \in \mathcal{S}}\int_{(\mathbf{0}_s, \mathbf{x}_s]}\frac{\diff \f_{\S}}{\diff P_s} \diff
    [\mathbb{P}_{s,n}-P_s]
    =
    n^{-1/2}\sum_{s \in \mathcal{S}}    \mathbb{G}_{s,n}{
      \left[
        \1_{(\mathbf{0}_s, \mathbf{x}_s]}\frac{\diff \f_{\S}}{\diff P_s}
      \right]},
  \end{equation*}
  where \(\mathbb{G}_{s,n}\) denotes the empirical process of the \(s\)-section
  of the data. As
  \(\{\1_{(\mathbf{0}_s, \mathbf{x}_s]} : \mathbf{x}_s \in (0, 1]^{|s|} \}\) is
  a Donsker class \citep{van1996weak}, it follows from the preservation
  properties of Donsker classes and the assumption that
  \( \1_{(\mathbf{0}_s, \mathbf{1}_s]}\diff \f_{\S} / \diff P_s \) is uniformly
  bounded, that also
  \begin{equation*}
    \mathcal{F}_s^* =
    \left\{
      \1_{(\mathbf{0}_s, \mathbf{x}_s]}\frac{\diff \f_{\S}}{\diff P_s} : \mathbf{x}_s \in (0, 1]^{|s|}
    \right\}
  \end{equation*}
  is a Donsker class. As this holds for any section \( \S \in \mathcal{S} \), we
  have
  \begin{equation*}
    \Vert \f_n - \f \Vert_{\infty} \leq
    n^{-1/2} \sum_{s \in \mathcal{S}}
    \sup_{f \in \mathcal{F}_s^*}
    |\mathbb{G}_{s,n}{[f]}|
    =     n^{-1/2} \sum_{s \in \mathcal{S}} \bigO_P{(1)}
    = \bigO_P(n^{-1/2}),
  \end{equation*}
  which in particular shows equation~(\ref{eq:tilde-key-prop}). To show
  equation~(\ref{eq:tilde-key-prop2}), note that
  \begin{equation}
    \label{eq:37}
    \f_n(\mathbf{x}) = \f(\mathbf{0}) + 
    \sum_{\S \in \mathcal{S}}
    \frac{1}{n} \sum_{i=1}^{n} \1_{(\mathbf{0}_s,
      \mathbf{1}_s]}(X_{s,i})\frac{\diff \f_{\S}}{\diff P_s}(X_{s,i})
    \1{\{X_{s,i} \preceq \mathbf{x}_{s}\}}.
  \end{equation}
  Equation~(\ref{eq:37}) shows that \(\f_n \in \D_n\), and by
  equation~(\ref{eq:8})
  \begin{align*}
    \Vert \f_{n} \Vert_{v}
    & = \f(\mathbf{0}) + \sum_{s \in \mathcal{S}} \frac{1}{n}\sum_{i=1}^{n}
      \1_{(\mathbf{0}_s, \mathbf{1}_s]}(X_{s,i}) \left\vert \frac{\diff
      \f_{\S}}{\diff P_s} \right\vert (X_{s,i})
    \\
    & = \f(\mathbf{0}) + \sum_{s \in \mathcal{S}}
      \mathbb{P}_{s,n}
      {\left[
      \1_{(\mathbf{0}_s, \mathbf{1}_s]}
      \frac{\vert \diff \f_{\S} \vert}{\diff P_s} 
      \right]},
  \end{align*}
  where we use \( |\diff \f_{\S}| / \diff P_s \) to denote the Radon-Nikodym
  derivative of \( |\mu_{\f_{\S}}| \) with respect to \( P_s \) on
  \( (\mathbf{0}_s, \mathbf{1}_s] \). The last equality follows from the
  properties of the Jordan-Hahn decomposition and the fact that \(P_s\) is a
  positive measure. By Assumption~\ref{assum:f0}~\ref{item:3} and the law of
  large numbers this implies that
  \(\Vert \f_{n} \Vert_{v} \arrow{P} \Vert \f \Vert_{v}\). As
  \(\Vert \f \Vert_{v}< M\) by Assumption~\ref{assum:f0}~\ref{item:11} it
  follows that \(P(\Vert \f_{n} \Vert_{v} < M) \rightarrow 1\), which shows
  equation~(\ref{eq:tilde-key-prop2}).

  Define the indicator variable
  \( \eta_n = \1{ \left\{ \f_n \in \D_{M,n} \right\}} \). Note that
  equation~(\ref{eq:tilde-key-prop2}) implies that \(P(\eta_n=1) \rightarrow 1\)
  and thus \citep[e.g.,][Lemma 2]{schuler2023selectively} yields
  \begin{equation}
    \label{eq:10}
    (1-\eta_n) = \smallO_P(a_n^{-1}) \quad \text{for any sequence }
    a_n \longrightarrow \infty.
  \end{equation}
  When \(\eta_n =1\) we have \(\f_n \in \D_{M,n}\) an thus by definition of
  \( \pi_n(\f) \) we have
  \( \eta_n \Vert \pi_n(\f) - \f \Vert_{\leb} \leq \eta_n \Vert \f_n - \f
  \Vert_{\leb} \). From this it follows that
  \begin{align*}
    \Vert \pi_n(\f) - \f \Vert_{\leb}
    & \leq \eta_n \Vert \f_n - \f \Vert_{\leb}
      + (1-\eta_n) \Vert \pi_n(\f) - \f \Vert_{\leb}
    \\
    & \leq  \Vert \f_n - \f \Vert_{\leb} + (1-\eta_n)2M
      = \bigO_P{(n^{-1/2})},
  \end{align*}
  where the last equality follows from equations~(\ref{eq:tilde-key-prop})
  and~(\ref{eq:10}).
\end{proof}

\section{Right-censored data}
\label{sec:censored-data}

Let \( T \in \R_+ \) be a time to event variable and \( W \in [0,1]^{d-1} \) a
covariate vector. In this section we discuss estimation of the hazard function
\( \alpha(t,\mathbf{w})\), for \( t \in [0,1] \) and
\( \mathbf{w} \in [0,1]^{d-1} \), where
\begin{equation*}
  \alpha(t, \mathbf{w}) = \lim_{\epsilon \searrow 0} \frac{P(T \in [t, t+ \epsilon] \mid T \geq t, W=\mathbf{w})}{\epsilon}.
\end{equation*}
We parameterize the log-hazard function as a multivariate càdlàg function with
bounded sectional variation norm,
\begin{equation}
  \label{eq:31}
  \log \alpha(t, \mathbf{w}) = f(t,\mathbf{w}),
  \quad \text{with} \quad
  f \in \D_{M}.
\end{equation}
Let \( C \in \R_+ \) be a right-censoring time. We assume conditional
independent censoring, i.e., \( C \independent T \mid W \). As we are only
interested in the conditional hazard function for \( t \in [0,1] \), we can
focus on the truncated event time \( T \wedge 1 \). We observe
\(O=(W,\tilde{T}, \Delta)\), where \( \tilde{T} = T \wedge 1\wedge C \) and
\( \Delta = \1{\{T \leq (C \wedge 1) \}} \). The right-censored data fits into
the setup described in Section~\ref{sec:empir-risk-minim} by setting
\( X = (W, \tilde{T}) \), \( Y=\Delta \), and \( \mathcal{Y} = \{0,1\} \). We
denote by \( n' \) the number of unique time points, and by
\(\tilde{T}_{(1)} < \tilde{T}_{(2)} < \tilde{T}_{(n')}\) the ordered sequence of
observed unique time points. We define \(\tilde{T}_{(0)} = 0\).

As loss function we use the negative log of the partial likelihood for \( f \)
\citep{cox1975partial,andersen2012statistical},
\begin{equation}
  \label{eq:nll-loss-survival}
  L^{\mathrm{pl}}(f, O) = \int_0^{\tilde{T}}e^{f(u, W)}  \diff u - \Delta f(\tilde{T}, W).
\end{equation}

The remainder of this section is organized as follows. We start by showing that
the empirical risk minimizer according to the partial likelihood loss is either
not defined or not consistent. We then show that the HAL estimator is
well-defined and derive its asymptotic convergence rate.

Proposition~\ref{prop:emp-risk-min-surv-dont-exist} gives a formal statement of
the problem described in Figure~\ref{fig:emp-risk-dens-problem} in
Section~\ref{sec:introduction}. To demonstrate the problem it is sufficient to
consider the univariate case without covariates.

\begin{proposition}  
  \label{prop:emp-risk-min-surv-dont-exist}
  Let \( f^{\circ} \in \DD{1}_M \) be given. If there exists a
  \( j \in \{1, \dots, n'-1\} \) such that
  \( f^{\circ}(\tilde{T}_{(j)}) > f^{\circ}(\tilde{T}_{(j+1)}) \), then
  \begin{equation*}
    f^{\circ} \not \in \argmin_{f \in \DD{1}_{M}}\empmeas{[L^{\mathrm{pl}}(f, \blank)]}.
  \end{equation*}
\end{proposition}
\begin{proof}
  See Appendix~\ref{sec:negat-log-likel}.
\end{proof}

Proposition~\ref{prop:emp-risk-min-surv-dont-exist} implies that any estimator
\( \hat{f}_n \in \DD{1}_M \) of the log-hazard function which decreases between
two time points is not an empirical risk minimizer. Thus, unless the hazard
function that generated the data is non-decreasing, an empirical risk minimizer
either does not exist or is inconsistent. Proposition~\ref{prop:hal-surv-exists}
on the other hand shows that a HAL estimator can be found as the solution to a
convex optimization problem.
\begin{proposition}
  \label{prop:hal-surv-exists}
  Let \( f_{\beta, n} \) be the data-dependent model defined in
  equation~(\ref{eq:9}). The problem
  \begin{equation}
    \label{eq:32}
    \min_{\Vert \beta\Vert_{1} \leq M} \empmeas{[L^{\mathrm{pl}}(f_{\beta, n}, \blank)]},
  \end{equation}
  is convex and has a solution. For any solution \( \hat{\beta} \),
  \( f_{\hat{\beta},n} \) is a HAL estimator, i.e.,
  \begin{equation*}
    f_{\hat{\beta},n} \in \argmin_{f \in \D_{M,n}}\empmeas{[L^{\mathrm{pl}}(f, \blank)]}.
  \end{equation*}
\end{proposition}
\begin{proof}
  See Appendix~\ref{sec:negat-log-likel}.
\end{proof}

We assume that the conditional hazard function for the right-censoring time
exists on \( [0,1) \) for all \( \mathbf{w} \in [0,1]^{d-1} \) and denote it by
\( \gamma(t, \mathbf{w}) \). We assume that \( \gamma \) is uniformly bounded
for all \( (t, \mathbf{w}) \in [0,1) \times [0,1]^{d-1} \). Without loss of
generality we can take
\begin{equation}
  \label{eq:33}
  P(\tilde{T} =1 \mid W= \mathbf{w}) = P(\tilde{T} =1, \Delta=0 \mid W= \mathbf{w}) =
  \exp{
    \left\{
      - \int_{[0,1)}  \gamma_0(s, \mathbf{w})  \diff s
    \right\}}.
\end{equation}
As \( T \) and \( C \) are assumed conditionally independent given \( W \), any
two uniformly bounded conditional hazard functions \( \alpha \) and \( \gamma \)
together with a marginal distribution for the covariate vector \( W \) uniquely
determine a distribution \( P \) for the observed data \( O \) through
equation~(\ref{eq:33}). We write $\alpha_P$ and $\gamma_P$ for the two
conditional hazard functions corresponding to a distribution \( P \), and let
\( f_P=\log \alpha_P \). We assume that \( W \) has a Lebesgue density and
denote this with $\omega_P$.

\begin{lemma}
  \label{lemma:unique-min}
  Let \( P \) be a distribution such that $\|\gamma_P\|_{\infty} < \infty$,
  \( \epsilon <\omega_P < 1/\epsilon\), for some \( \epsilon >0 \), and
  \( f_P \in \D_M \). Then for all \( f \in \D_M \),
  \begin{equation*}
    P[{L^{\mathrm{pl}}(f, \blank)-L^{\mathrm{pl}}(f_P, \blank)}] \asymp \Vert f-f_P\Vert_{\leb}^2.
  \end{equation*}
\end{lemma}

\begin{proof}
  The lemma essentially follow from general properties of the Kullback-Leibler
  divergence. However, due to the point-mass at \( t=1 \), a few additional
  arguments are needed which we present in Appendix~\ref{sec:negat-log-likel}.
\end{proof}

\begin{corollary}
  \label{cor:main-survival}
  Let \( P \) be a distribution such that $\|\gamma_P\|_{\infty} < \infty$,
  \( \epsilon <\omega_P < 1/\epsilon\), for some \( \epsilon >0 \), and let
  $\hat{f}_n$ be a HAL estimator based on the negative partial log-likelihood
  loss defined in equation~(\ref{eq:nll-loss-survival}). If
  Assumption~\ref{assum:f0} holds, then
  \begin{equation*}
    \Vert \hat{f}_n - f_P \Vert_{\leb} = \smallO_P(n^{-1/3}\log(n)^{2(d-1)/3}).
  \end{equation*}
\end{corollary}
\begin{proof}
  Corollary~\ref{cor:main-survival} follows from
  Theorem~\ref{theorem:main-result} and Lemma~\ref{lemma:unique-min}. Details
  are given in Appendix~\ref{sec:negat-log-likel}.
\end{proof}

We illustrate the HAL estimator of a conditional hazard function and the effect
of the sectional variation with the following example. Consider a study that
enrolls patients between the age of 20 and 60 to study the effect of a treatment
on death within one year after treatment. We simulate an artificial dataset such
that the hazard of death does not depend on age in the untreated group, while
the hazard of death among treated patients is lowered for patients younger than
40, but increased for older patients. Censoring is generated independently of
covariates and event times. As noted by \cite{rytgaard2021estimation}, the loss
in equation~(\ref{eq:14}) can be recognized as the negative log-likelihood of a
Poisson model. This implies that we can use existing software from the
\texttt{R}-packages \texttt{glmnet}
\citep{Friedman_Tibshirani_Hastie_2010_Regularizat_Paths_Generalized,Tay_Narasimhan_Hastie_2023_Elastic_Net_Regularizat}
and \texttt{hal9001}
\citep{hejazi2020hal9001,Coyle_Hejazi_Phillips_Laan_Laan_2022} to construct a
HAL estimator. The HAL estimator, computed on a simulated dataset of 200
patients, is displayed for the treated group in Figure~\ref{fig:hal-hazard}
across various values of the sectional variation norm \( M \). We illustrate the
corresponding estimate of the conditional survival function for both treatment
groups in Figure~\ref{fig:hal-survfun}.

\begin{figure}
  \centerline{\includegraphics[width=1\linewidth]{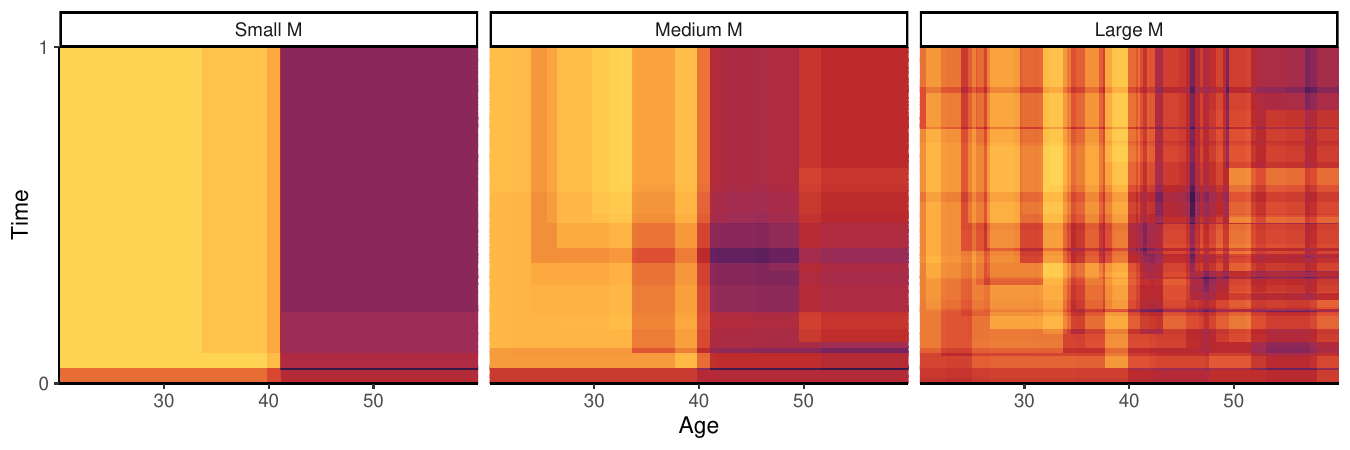}}
  \caption[]{The HAL estimator of the hazard function for the treated group
    based on a sample size of 200 from the simulated study with darker values
    corresponding to higher values of the hazard function. Estimates are shown
    for three different values of the sectional variation norm ($M$). }
  \label{fig:hal-hazard}
\end{figure}
  
\begin{figure}
  \centerline{\includegraphics[width=1\linewidth]{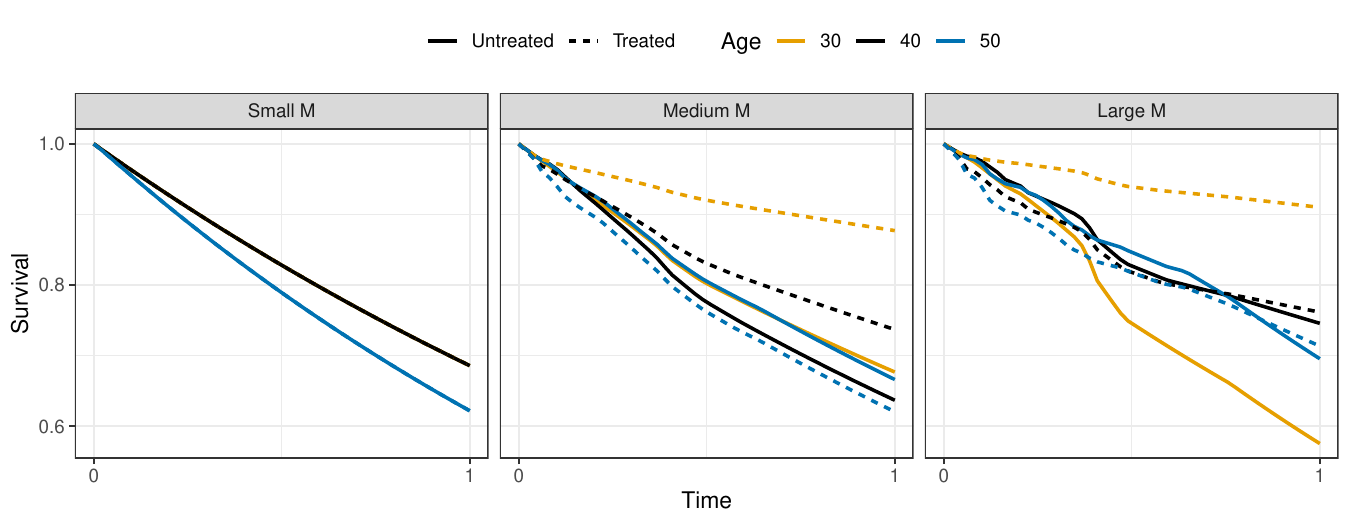}}
  \caption[]{Estimates of the survival function derived from the HAL estimator
    stratified on treatment and three different age values based on a sample of
    200 patients from the simulated study. Estimates are shown for three
    different values of the sectional variation norm ($M$).}
  \label{fig:hal-survfun}
\end{figure}

\section{Density estimation}
\label{sec:density-estimation}

Let \( U \in [0,1] \) and \( W \in [0,1]^{d-1} \) and consider estimation of the
conditional density of \( U \) given \( W \). In this section the available data
are \( O=(U,W) \), i.e., in the notation of the general setup of Section
\ref{sec:empir-risk-minim}, \(X=(U,W)\) and no additional variable \( Y \) is
observed. We parameterize the conditional density as an element of
\begin{equation}
  \label{eq:dens-par-hal}
  \mathcal{P}^d_M = 
  \left\{
    p \colon [0,1]^d \rightarrow \R_+ \midd
    \log p(u, \mathbf{w}) = f(u, \mathbf{w}) - \log
    \left(
      \int_0^1 e^{f(z, \mathbf{w})} \diff z
    \right),
    f \in \D_M
  \right\}.
\end{equation}
This parametrization is a natural one and has been used before for (univariate)
density estimation
\citep[e.g.,][]{leonard1978density,silverman1982estimation,gu1993smoothing}.
Note that any element of \( \mathcal{P}^d_M \) is a conditional density, and
that \( \mathcal{P}^d_M \) includes all conditional densities \( p \) such that
\( \log p \in \D_M\). Define the data-adaptive model
\begin{equation*}
  \mathcal{P}^d_{M,n} = 
  \left\{
    p \in \mathcal{P}_M^d \midd
    \log p(u, \mathbf{w}) = f(u, \mathbf{w}) - \log
    \left(
      \int_0^1 e^{f(z, \mathbf{w})} \diff z
    \right),
    f \in \D_{M,n}
  \right\},
\end{equation*}
and a HAL estimator as
\begin{equation}
  \label{eq:13}
  \hat{p}_n \in \argmin_{p \in \mathcal{P}^d_{M,n}} \empmeas{[-\log p]}.
\end{equation}

Proposition~\ref{prop:dens-convex} shows that a HAL estimator is well-defined
and can be found as the solution to a convex optimization problem.
\begin{proposition}
  \label{prop:dens-convex}
  Define the set of indices
  \( \mathcal{I} = \left\{ \{1\} \cup s : s \subset \{2, \dots, d\} \right\} \)
  and let
  \begin{equation*}
    g_{\beta, n}(\mathbf{x})
    = \sum_{i=1}^{n}\sum_{r \in \mathcal{I}}\beta_{i}^r\1{\{X_{r, i} \preceq \mathbf{x}_{r}\}},
    \quad \text{with} \quad    
    \beta = \{\beta^{r} = (\beta_{1}^r, \dots,\beta_{n}^r) : r \in \mathcal{I} \}.
  \end{equation*}
  The problem
  \begin{equation}
    \label{eq:20}
    \min_{\Vert \beta\Vert_1 \leq M} \empmeas{
      \left[
        \bar{L}(g_{\beta, n}, \blank)
      \right]},
    \quad \text{with} \quad
    \bar{L}(g, O) =   \log
    \left(
      \int_0^1 e^{g(z, W)} \diff z
    \right)
    - g(U, W),
  \end{equation}
  is convex and has a solution. For any solution \( \hat{\beta} \),
  \begin{equation*}
    p_{\hat{\beta},n} \in \argmin_{p \in \mathcal{P}^d_{M,n}}\empmeas{[-\log
      p]},
  \end{equation*}  
  where
  \begin{equation*}
    \log p_{\hat{\beta},n}(u,\mathbf{w}) = g_{\hat{\beta}, n}(u,\mathbf{w}) - \log
    {\left(
        \int_0^1 e^{g_{\hat{\beta}, n}(z,\mathbf{w})} \diff z
      \right)}.
  \end{equation*}
\end{proposition}
\begin{proof}
  See Appendix~\ref{sec:density-estimation-1}
\end{proof}
Proposition \ref{prop:dens-convex} shows that the HAL estimator defined in
equation~(\ref{eq:13}) does not need to include basis functions that are only
functions of \( w \), so the number of basis functions is reduced to
\( |\mathcal{I}|=n2^{d-1} \).

We assume that \( (U,W) \sim P \) for some distribution
\( P \ll \leb \). For a distribution \( P \), let \( p_P \) denote the
conditional density of \( U \) given \( W \) and \( \omega_P \) the
marginal density of \( W \) with respect to \( \lambda \).

\begin{corollary}
  \label{cor:hal-dens}
  Let \( P \) be a distribution such that
  \( \epsilon < \omega_P  < 1/\epsilon \), for some $\epsilon>0$, and
  \( p_P \in \mathcal{P}_M^d \), and let $\hat{p}_n$ be a HAL estimator as
  defined in equation~(\ref{eq:13}). If Assumption~\ref{assum:f0} holds when
  \( \f \) is the minimizer of \( f \mapsto P{[\bar{L}(f, \blank)]} \) over
  \( \D_M \), then
  \begin{equation*}
    \Vert \hat{p}_n - p_P \Vert_{\leb} = \smallO_P(n^{-1/3}\log(n)^{2(d-1)/3}).
  \end{equation*}
\end{corollary}
\begin{proof}
  Corollary~\ref{cor:hal-dens} follows from Theorem~\ref{theorem:main-result}.
  See Appendix~\ref{sec:density-estimation-1} for a detailed proof.
\end{proof}

A density can be obtained from a hazard function. This implies that an
alternative density estimator can be constructed by first using the HAL
estimator defined in Section~\ref{sec:censored-data} to estimate the
corresponding log-hazard function and then transforming this into a density. We
refer to this estimator as a `HAL hazard parametrization' and to the estimator
defined in equation~(\ref{eq:13}) as a `HAL density parametrization'. We compare
these two estimators in Figure~\ref{fig:density-pars}, where we have fitted both
estimators to a simulated univariate dataset. The estimators are implemented
using the convex optimization package \texttt{CVXR} in \texttt{R}
\citep{Fu_Narasimhan_Boyd_2020_Cvxr} and the bounds \( M \) on the sectional
variation norms are selected using cross-validation. We see that the estimator
based on the hazard parametrization can exhibit erratic behavior at the end of
the interval. The reason is that assuming a log-hazard function belongs to
\( \D_M \) implies that the corresponding density will not integrate to one. To
see this, observe that the conditional survival function associated with a
log-hazard function \(f\in\D_M\) evaluated at \( t=1 \) is
\begin{equation*}
  \exp{
    \left\{
      - \int_0^1 e^{f(z, \mathbf{w})} \diff z
    \right\}}
  \geq \exp{
    \left\{
      -e^M
    \right\}}>0.
\end{equation*}
Thus when the support of \( U \) is \( [0,1] \), the assumption that the
log-hazard belongs to \( \D_M \) will be wrong by definition for any
\( M < \infty \). We argue that the parametrization in
equation~(\ref{eq:dens-par-hal}) is better suited when \( U \) is known to have
support in \( [0,1] \).

\begin{figure}
  \centerline{\includegraphics[width=1\linewidth]{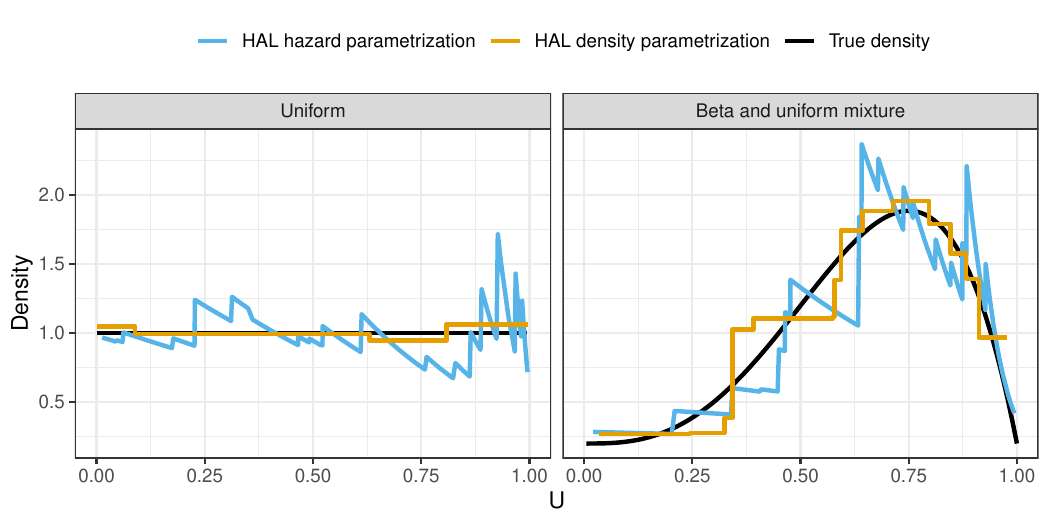}}
  \caption[]{Two different density estimators under two different
    data-generating distributions. We generated 200 samples from a uniform
    distribution (left panel) and a mixture of a beta distribution and a uniform
    distribution (right panel). The `HAL hazard parametrization' refers to a
    density estimator obtained from a HAL estimator of a hazard function, which
    was defined in Section~\ref{sec:censored-data}. The `HAL density
    parametrization' refers to the HAL estimator defined in
    equation~(\ref{eq:13}).}
  \label{fig:density-pars}
\end{figure}

\section{Least-squares regression}
\label{sec:regression-function}

Let \( O = (X, Y) \) for \( X \in [0,1]^d \) and \( Y \in [-B,B] \) for some
\( B < \infty \), and define
\begin{equation*}
  f_P(\mathbf{x}) = \E{\left[ Y \mid X=\mathbf{x} \right]},
  \quad \text{when} \quad (X,Y) \sim P.
\end{equation*}
In this section we consider estimation of \( f_P \) using the squared error loss
\begin{equation}
  \label{eq:35}
  L^{\mathrm{se}}(f, O) = (f(X) - Y)^2.
\end{equation}
We here use \( \omega_P \) to denote the Lebesgue density of \( X \) which we
assume to exist.

\begin{corollary}
  \label{cor:least-squar-hal}
  Let \( P \) be a distribution such that
  \( \epsilon < \omega_P < 1/\epsilon \), for some $\epsilon>0$, and
  \( f_P \in \D_M \), and let $\hat{f}_n$ be a HAL estimator based on the
  squared error loss defined in equation~(\ref{eq:35}). If
  Assumption~\ref{assum:f0} holds, then
  \begin{equation*}
    \Vert \hat{f}_n - f_P \Vert_{\leb} = \smallO_P(n^{-1/3}\log(n)^{2(d-1)/3}).
  \end{equation*}
\end{corollary}
\begin{proof}
  We show that Assumption~\ref{assum:loss-assum} holds for the squared error
  loss, and so Corollary~\ref{cor:least-squar-hal} follows from
  Theorem~\ref{theorem:main-result}. First note that because the squared error
  loss is a strictly proper scoring rule \citep{gneiting2007strictly}, the
  assumption that \( f_P \in \D_M \) implies that \( \f = f_P \) a.e.
  Conditions~\ref{assum:loss-assum}~\ref{item:1}-\ref{item:2} hold by the
  definition of the squared error loss and the assumption that \(Y \) and
  $\omega_P$ are bounded. Condition~\ref{assum:loss-assum}~\ref{item:5} holds by
  Proposition 3 and Lemma~4 in Appendix~B of \citep{bibaut2019fast}.
\end{proof}

For the squared error loss an empirical risk minimizer as defined in
equation~(\ref{eq:emp-full-min}) exists. This was formally shown by
\cite{fang2021multivariate}. The authors also derive an algorithm for finding a
collection of basis functions that is sufficient to construct an empirical risk
minimizer. We illustrate the difference between the HAL estimator and the
empirical risk minimizer by comparing the number of basis functions needed to
calculate the two estimators for different sample sizes and dimensions. The
results are shown in Figure~\ref{fig:n-basis-func}. We see that a HAL estimator
can be constructed using much fewer basis functions.

\begin{figure}
  \centerline{\includegraphics[width=1\linewidth]{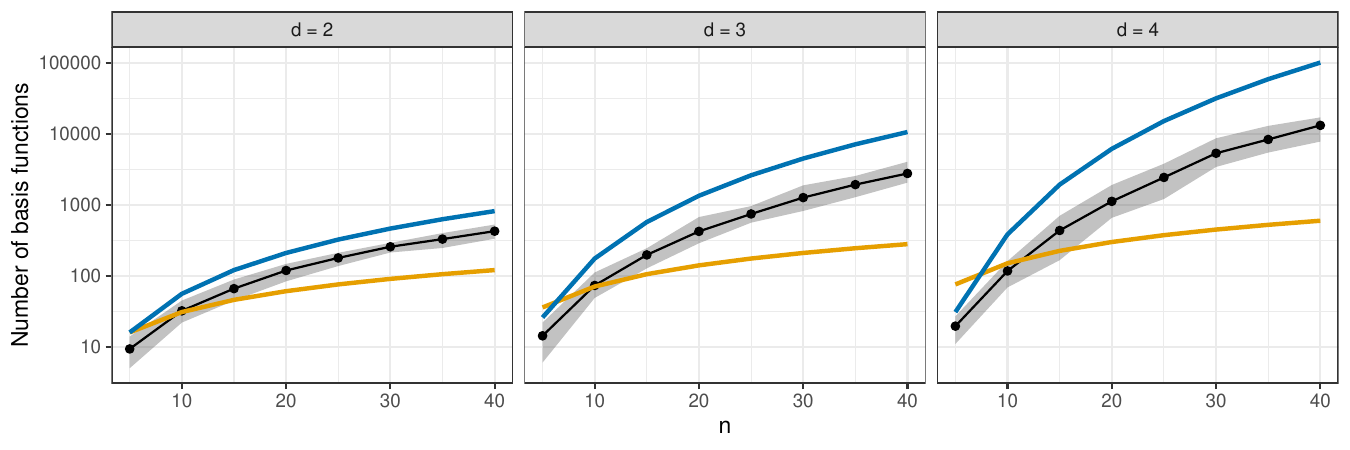}}
  \caption[]{The black line is the average number of basis functions needed to
    calculate the empirical risk minimizer with ribbons denoting the 2.5\%- and
    97.5\%-quantiles based on 200 simulations of uniformly distributed
    covariates. The number of observations is denoted by \( n \) and the
    dimension by \( d \). The blue line is a deterministic upper bound on this
    number (see Lemma 3.5 of \cite{fang2021multivariate}). The orange line is
    the number of basis functions needed to calculate a HAL estimator. }
  \label{fig:n-basis-func}
\end{figure}

\section{Discussion}
\label{sec:discussion}

Our main result relies on the smoothness assumption~\ref{assum:f0}~\ref{item:3}.
However, another HAL estimator could be defined using a finer sieve than the one
we have chosen, and we expect that Assumption~\ref{assum:f0}~\ref{item:3} can be
relaxed. An interesting question is how much we can reduce the number of basis
functions and still achieve the same rate of convergence, and whether we need to
impose additional smoothness assumptions for this to hold.

Throughout this paper we have stated that an empirical risk minimizer does not
exist or is inconsistent when a density or a hazard function is estimated.
Formally, our Proposition~\ref{prop:emp-risk-min-surv-dont-exist} does not rule
out, however, that a consistent empirical risk minimizer can exist in the special case
that the data-generating hazard function is non-decreasing. If the
data-generating hazard function is believed to be monotone, it is natural to use
shape-constrained estimators \citep{groeneboom2014nonparametric}. An interesting
direction for future research is to investigate HAL-like estimators when
biologically motivated monotonicity constraints are imposed.

\begin{appendix}
  
\section{Càdlàg functions and measures}
\label{sec:cadl-funct-meas}

To prove the results from Section~\ref{sec:funct-bound-sect}, we start by
proving the following two lemmas.
\begin{lemma}
  \label{lemma:unif-limit}
  For a function \( f \colon [0,1]^d \rightarrow \R \) and a sequence of
  functions \( f_n \colon [0,1]^d \rightarrow \R\), \( n \in \N \), assume that
  \( \| f_n -f\|_{\infty} \rightarrow 0\) when \( n \rightarrow \infty \). If
  \( f_n \in \D_M \) for all \( n \in \N \) then \( f \in \D_M\).
\end{lemma}
\begin{proof}
  \cite{neuhaus1971weak} shows that the uniform limit of a sequence of càdlàg
  functions is also càdlàg. It thus only remains so be shown that
  \( \| f \|_v \leq M \). Assume for contradiction that this is not the case. We
  thus assume that \( \| f \|_v > M + \epsilon \) for some \( \epsilon>0 \),
  which by definition means that there must exist finite partitions
  \( \mathcal{P}_{\S} \) of all faces \( (\mathbf{0}_{\S}, \mathbf{1}_{\S}] \),
  \( \emptyset \not =\S \subset [d] \) such that
  \begin{equation*}
    \sum_{\S \in \mathcal{S}} \sum_{A \in \mathcal{P}_s} |\Delta(f; A)| > M +
    \epsilon,
    \quad \text{where} \quad 
    \mathcal{S} = \{ s \subset [d] : s \not = \emptyset\} 
  \end{equation*}
  The sum above is made up of
  \( \kappa = \sum_{s} |\mathcal{P}_s|2^{|s|}< \infty \) number of terms on the
  form \( \pm f(\mathbf{x}) \) for some \( \mathbf{x} \in [0,1]^d \). By
  assumption we can find \( n_0 \in \N \) such that
  \( \|f_n-f\|_{\infty}< \epsilon/\kappa \) for all \( n \geq n_0 \), and thus
  \begin{equation*}
    M < \sum_{\S \in \mathcal{S}} \sum_{A \in \mathcal{P}_{\S}} |\Delta(f_n; A)| \leq \| f_n \|_v,
    \quad \forall n > n_0.
  \end{equation*}
  This contradicts the fact that \( f_n \in \D_M \) for all \( n \in \N \), so
  we must have \( \| f \|_v \leq M \).
\end{proof}

\begin{lemma}
  \label{lemma:DKF-argument}
  Let \( f \) be a function that is right-continuous in each of its coordinates
  with \( \| f \|_v \leq M \). There exists a sequence
  \( \{f_n\} \subset \mathcal{R}^d_M \) such that
  \( \Vert f- f_n \Vert_{\infty} \rightarrow 0\) for \( n \rightarrow \infty\).
\end{lemma}

\begin{proof}
  By Theorem~3~(a) in \citep{ChristophAistleitner2015} there exists a unique,
  finite signed measure \( \mu_f \) such that
  \( f(\mathbf{x})=\mu([\mathbf{0},\mathbf{x}]) \). By the Jordan-Hahn
  decomposition theorem we may write \( \mu_f = \alpha P^+ - \beta P^- \), where
  \( P^+ \) and \( P^- \) are uniquely determined probability measures with
  \( P^+ \perp P^- \), and \( \alpha, \beta \in [0,\infty) \). Letting \( F^+ \)
  and \( F^- \) denote the associated cumulative distribution functions, we have
  that \( f = \alpha F^+ - \beta F^- \). By Theorem~3~(a) in
  \citep{ChristophAistleitner2015} and because \( P^+ \perp P^- \) we have
  \begin{equation}
    \label{eq:18}
    M \geq \| f \|_v = \| \mu_f \|_{\mathrm{TV}} =
    \alpha \| P^+ \|_{\mathrm{TV}} + \beta \| P^- \|_{\mathrm{TV}}
    =
    \alpha + \beta.
  \end{equation}
  Let \( P_n^+ \) and \( P_n^- \) denote the empirical measures obtained from
  i.i.d.\ samples from \( P^+ \) and \( P^- \), respectively. Let \( F_n^+ \)
  and \( F_n^- \) denote the associated empirical distribution functions, and
  define \( F_n = \alpha F_n^+ - \beta F_n^- \). As \( P_n^+ \perp P_n^- \)
  almost surely we have
  \begin{equation*}
    \| F_n \|_v
    =
    \alpha \| P^+_n \|_{\mathrm{TV}} + \beta \| P_n^- \|_{\mathrm{TV}}
    =
    \alpha + \beta
    \quad \text{a.s.}    
  \end{equation*}
  The multivariate version of the Dvoretzky-Kiefer-Wolfowitz theorem
  \citep{dvoretzky1956asymptotic,naaman2021tight} and the Borel-Cantelli lemma
  imply that \( \| F_n^+ -F^+\|_{\infty} \rightarrow 0 \) and
  \( \| F_n^- -F^-\|_{\infty} \rightarrow 0 \) almost surely. Hence there must
  exist deterministic sequences of discrete measures \( p^+_n \) and \( p^-_n \)
  with associated cumulative distribution functions \( f_n^+ \) and \( f_n^- \)
  such that
  \begin{equation}
    \label{eq:17}
    p_n^+ \perp p_n^-,
    \quad \forall n \in \N,
  \end{equation}
  and
  \begin{equation}
    \label{eq:16}
    \|f_n^+ -F^+\|_{\infty} \longrightarrow 0
    \quad \text{and} \quad 
    \|f_n^- -F^-\|_{\infty} \longrightarrow 0.
  \end{equation}
  Note that \( f_n^+ \) is a linear combination of the indicator functions
  \( \{\1_{[\mathbf{x}_i, \mathbf{1}]}\}_{i=1}^n \), where
  \( \{\mathbf{x}_i\}_{i=1}^n \) are the support points of the discrete measure
  \( p_n^+ \), and similarly for \( f_n^- \). Hence, with
  \( f_n = \alpha f_n^+ - \beta f_n^-\), we have that
  \( f_n \in \mathrm{Span}(\mathcal{F}^d) \). By equations~(\ref{eq:18})
  and~(\ref{eq:17}),
  \begin{equation*}
    \| f_n \|_v = \alpha + \beta \leq M,
  \end{equation*}
  so \( f_n \in \mathcal{R}_M^d \) for all \( n \in \N \).
  Equation~(\ref{eq:16}) gives that \( \| f_n - f\|_{\infty} \rightarrow 0\)
  which concludes the proof.
\end{proof}

\begin{proof}[Proof of Proposition~\ref{prop:discrete-measure-closure}]
  For \( f_1, f_2 \in \D \) and \( \alpha, \beta \in \R \) the function
  \( f = \alpha f_1 + \beta f_2 \) is càdlàg, so $\mathcal{R}_M^d \subset \D_M$.
  It thus follows from Lemma~\ref{lemma:unif-limit} that
  \( \overline{\mathcal{R}_M^d} \subset \D_M \). As any \( f \in \D_M\) is
  right-continuous in each of its coordinates, the reverse inclusion follows
  from Lemma~\ref{lemma:DKF-argument}.
\end{proof}

\begin{proof}[Proof of Proposition~\ref{prop:cadlag-measure}]
  Any function \( f \in \D_M \) is by definition right-continuous in each of its
  arguments so the first statement follows immediately from Theorem~3~(a) in
  \citep{ChristophAistleitner2015}. For the second statement, we know by
  Theorem~3~(b) in \citep{ChristophAistleitner2015} that there exists a
  right-continuous function \( f_{\mu} \) with \( \|f_{\mu}\|_v = M < \infty \)
  such that \( f_{\mu}(\mathbf{x}) = \mu([\mathbf{0},\mathbf{x}]) \). By
  Lemma~\ref{lemma:DKF-argument}, \( f_{\mu} \) can be approximated uniformly by
  a sequence of functions \( f_n \in \D_M \). Lemma~\ref{lemma:unif-limit} then
  implies that \( f_{\mu} \in \D_M \).
\end{proof}

\begin{proof}[Proof of Proposition~\ref{prop:cadlag-repr}]
  For the first statement we use that we can partition a box
  \([\mathbf{0},\mathbf{x}]\) into `half-closed' lower dimensional faces with
  corners at \( \mathbf{0} \), the point \( \mathbf{0} \), and the remaining
  `half-closed interior' of the box, i.e.,
  \begin{equation*}
    [\mathbf{0}, \mathbf{x}]
    = \{ \mathbf{0} \} \cup
    \Big(
    \bigcup_{\S \in \mathcal{S}} A(\mathbf{x}; s),
    \Big),
    \quad \text{for} \quad
    A(\mathbf{x}; s) = A_1(\mathbf{x}; s) \times \cdots \times A_d(\mathbf{x};
    s),
  \end{equation*}
  where
  \begin{equation*}
    \mathcal{S} =
    \left\{
      s \subset \{1, \dots, d\} : s\not= \emptyset
    \right\},
    \quad \text{and} \quad
        A_i(\mathbf{x}; s) = 
    \begin{cases}
      (0, x_i] & \text{if } i \in s \\
      \{0\} & \text{if } i \not\in s \\
    \end{cases},
  \end{equation*}
  and we define \( (0, 0 ] = \emptyset \).  
  Using this and Proposition~\ref{prop:cadlag-measure} we can write
  \begin{equation}
    \label{eq:4}
    f(\mathbf{x}) = \mu_f([\mathbf{0}, \mathbf{x}])
    = \mu_f(\{ \mathbf{0} \})
    + \sum_{\S \in \mathcal{S}} \mu_f(A(\mathbf{x}; s))
  \end{equation}    
  Any section \(f_s\) of \(f\) is also a càdlàg function with bounded sectional
  variation norm and hence generates a measure on the cube \([0,1]^{|s|}\)
  through the relation
  \begin{equation}
    \label{eq:2}
    f_s(\mathbf{x}) = \mu_{f_s}([\mathbf{0}_{\S}, \mathbf{x}]),
    \quad \text{for all} \quad \mathbf{x} \in [0,1]^{|s|}.
  \end{equation}
  By definition of the section \(f_s\) it follows that the measure assigned to a
  box in \([0,1]^{|s|}\) by \(\mu_{f_s}\) is the same as the measure assigned by
  \(\mu_f\) when this space is considered as a subspace of \([0,1]^d\), i.e.,
  \begin{equation*}
    \mu_{f_s}([\mathbf{0}_{\S}, \mathbf{x}_{\S}]) = \mu_{f}([\mathbf{0}, \overline{\mathbf{x}}_{\S}]), \quad \text{for} \quad \mathbf{x} \in [0,1]^{d}.
  \end{equation*}
  By the uniqueness of the measures generated by \( f \) and each \( f_s \) it
  follows that
    \begin{equation}
    \label{eq:3}
    \mu_{f}(A(\mathbf{x}; \S)) = \mu_{f_s}((\mathbf{0}_{\S}, \mathbf{x}_{\S}]).
  \end{equation}
  By equations~(\ref{eq:4})~and~(\ref{eq:3}) we then have
  \begin{align*}
    f(\mathbf{x})
    = f(\mathbf{0})
    + \sum_{\S \in \mathcal{S}} \mu_{f_s}((\mathbf{0}_s, \mathbf{x}_s])
    = f(\mathbf{0})
    + \sum_{\S \in \mathcal{S}} \int_{(\mathbf{0}_s, \mathbf{x}_s]} \diff f_s.
  \end{align*}
  The second statement follows because \( A(\mathbf{1}; \S) \), for
  \( \S \in \mathcal{S} \), are disjoint sets, so the measures
  \( \1_{A(\mathbf{1}; \S)}\cdot\mu_f \) are mutually singular. Hence,
  \begin{align*}
    \| \mu_f \|_{\mathrm{TV}}
    & =
      \left\| \1_{\{\mathbf{0}\}}\cdot \mu_f + \sum_{\S \in \mathcal{S}}
      \1_{A(\mathbf{1}; \S)}\cdot\mu_f \right\|_{\mathrm{TV}}
    \\
    & =
      \left\| \1_{\{\mathbf{0}\}}\cdot \mu_f \right\|_{\mathrm{TV}} +
      \sum_{\S \in \mathcal{S}}
      \left\|
      \1_{A(\mathbf{1}; \S)}\cdot\mu_f
      \right\|_{\mathrm{TV}} 
    \\
    & = \int_{\{\mathbf{0}\}}  \diff
      |\mu_f|
      +
      \sum_{\S \in \mathcal{S}} \int_{A(\mathbf{1}; \S)}  \diff
      |\mu_f|
    \\
    & = | f(\mathbf{0})| +      
      \int_{(\mathbf{0}_{\S}, \mathbf{1}_{\S}]}  | \diff f_s|. 
  \end{align*}
\end{proof}

\begin{proof}[Proof of Proposition~\ref{prop:pw-cons-cadlag}]
  Let \( B_{r}(\mathbf{x}) \) be the ball around the point
  \( \mathbf{x} \in [0,1]^d \) with radius \( r>0 \). For a function
  \( f : [0,1]^d \rightarrow \mathcal{K} \subset \R \) with \( \mathcal{K} \)
  finite, we now claim that
  \begin{equation}
    \label{eq:36}\tag{\( * \)}    
    \forall \mathbf{x} \in [0,1]^d, \forall \mathbf{a}  \in
    \{0,1\}^d, \exists r >0, \forall z,y  \in B_{r}(\mathbf{x}) \cap
    Q_{\mathbf{a}}(\mathbf{x}) : f(z)=f(y),
  \end{equation}
  implies \( f \in \mathcal{R}^d_M \). To see this, assume that (\ref{eq:36})
  holds. Define the covering
  \begin{equation*}
    \mathcal{B} =  \{B(\mathbf{x}) : \mathbf{x} \in [0,1]^d\} ,
  \end{equation*}
  where \( B(\mathbf{x}) \) is an open ball around \( \mathbf{x} \) such that
  for any \( \mathbf{a} \in \{0,1\}^d \), \( f \) is constant on
  \( B_{r_{\mathbf{x}}}(\mathbf{x}) \cap Q_{\mathbf{a}}(\mathbf{x}) \) for some
  \( r_{\mathbf{x}} > 0 \). Such an
  open ball exists around any \( \mathbf{x} \) by (\ref{eq:36}). As
  \( [0,1]^d \) is compact there exists a finite subset
  \( \{B(\mathbf{x}^1), \dots, B(\mathbf{x}^J)\} \subset \mathcal{B} \) that
  covers \( [0,1]^d \). Consider now any box of the form
  \begin{align*}
    I(\mathbf{j}) = I_1(j_1) \times \cdots I_d(j_d),
    \quad \text{where} \quad
    I_i(j) = 
    \begin{cases}
      [0, x^{1}_i) & \text{if } j=1,\\
      [x^j_i, x^{j+1}_i) & \text{if } 0<j<J,\\
      [x^J_i 1] & \text{if } j=J,
    \end{cases}    
  \end{align*}
  for all unique sequences
  \(\mathbf{j}= (j_1, \dots, j_d)\in \{1, \dots, J\}^{d} \). These boxes
  partition \( [0,1]^d \), and by construction of the covering
  \( \{B(\mathbf{x}^1), \dots, B(\mathbf{x}^J)\} \), \( f \) is constant on
  \( B(\mathbf{x}^j) \cap I(\mathbf{j}) \) for all \( j \) and $\mathbf{j}$. As
  any \( I(\mathbf{j}) \) is connected and
  \( \{B(\mathbf{x}^1), \dots, B(\mathbf{x}^J)\} \) is an open cover, it follows
  that \( f \) is constant on each \( I(\mathbf{j}) \). Hence
  \( f \in \mathcal{R}_M^d \), and thus we have proved the initial claim. The
  proposition now follows by noting that this implies that if
  \( f \not \in \mathcal{R}_M^d \), then (\ref{eq:36}) is false, i.e.,
  \begin{equation*}
    \exists \mathbf{x} \in [0,1]^d, \exists \mathbf{a}  \in
    \{0,1\}^d, \forall r >0, \exists z,y  \in B_{r}(\mathbf{x}) \cap
    Q_{\mathbf{a}}(\mathbf{x}) : f(z)\not=f(y).
  \end{equation*}
  Thus, if \( f \not \in \mathcal{R}_M^d \), we can find a point
  \( \mathbf{x} \in [0,1]^d \), a vertex \( \mathbf{a} \in \{0,1\}^d \) and a
  sequence \( r_n \searrow 0\) such that for all \( n \in \N \), \( f \) is not
  constant on \( B_{r_n}(\mathbf{x}) \cap Q_{\mathbf{a}}(\mathbf{x}) \). This in
  turn implies that we can find a sequence
  \( \{\mathbf{x}_n \} \in B_{r_n}(\mathbf{x}) \cap Q_{\mathbf{a}}(\mathbf{x})
  \) such that \( f(\mathbf{x}_{n}) \not = f(\mathbf{x}_{n-1}) \). Clearly,
  \( \mathbf{x}_n \in Q_{\mathbf{a}}(\mathbf{x}) \) and
  \( \mathbf{x}_n \rightarrow \mathbf{x} \), but as
  \( f(\mathbf{x}) \in \mathcal{K}\) for all \( \mathbf{x} \in [0,1]^d \),
  \( f(\mathbf{x}_{n}) \) cannot converge. Hence \( f \) is not càdlàg.
\end{proof}

\section{Results from empirircal process theory}
\label{sec:results-from-empir}

Recall the notation
\( \mathcal{L}_{M} = \left\{ L(f, \blank) : f \in \D_M \right\} \) for a loss
function \(L \colon \D_M \times \mathcal{O} \rightarrow \R_+\), and the
Assumption~\ref{assum:loss-assum}~\ref{item:5}, which we restate her for
convenience:
\begin{equation}  
  \label{eq:recal-ass}\tag{B}
  \exists C < \infty , \eta >0, \kappa \in \N 
  :
  N_{[\,]}(\epsilon, \mathcal{L}_M, \| \blank \|_{P}) \leq C
  N_{[\,]}(\epsilon/C, \D_M, \| \blank \|_{\leb} )^{\kappa},
  \quad \forall  \epsilon\in(0,\eta).
\end{equation}
Define
\begin{equation}
  \label{eq:29}    
  \Gamma_n(\delta) =  \sup_{\Vert f - \f \Vert_{\leb} < \delta}
  \left\vert \mathbb{G}_{n}{[L(f, \blank) - L(\f, \blank)]} \right\vert
  \quad \text{with} \quad f \in \D_M,
\end{equation}
where \( \mathbb{G}_n = \sqrt{n}(\empmeas- P)\) is the empirical process.
\begin{lemma}
  \label{lemma:bound-emp-brack}
  If (\ref{eq:recal-ass}) holds and \( \Vert L(f, \blank) \Vert_{\infty} <C \),
  then there exists an $\eta>0$ such that for for all \( n \in \N \) and
  $\delta \in (0, \eta)$,
  \begin{equation*}
    \E^*_P{\left[ \Gamma_n(\delta) \right]}     \lesssim \delta^{1/2}
    |\log(\delta)|^{d-1}
    + \frac{|\log(\delta)|^{2(d-1)} }{\delta  \sqrt{n}}.
  \end{equation*}
  In particular, when \( r_n = n^{1/3}\log(n)^{-2(d-1)/3}\) we have
  \begin{equation*}
    n^{-1/2}\E^*_P{\left[ \Gamma_n(r_n ) \right]} =  \bigO{(r_n^{-2})}.
  \end{equation*}
\end{lemma}

\begin{proof}
  Define the entropy integral
  \begin{equation*}
    J_{[\,]}(\delta, \mathcal{H}, \Vert \blank \Vert) = \int_0^{\delta}
    \sqrt{1 + \log N_{[\,]}(\epsilon,\mathcal{H}, \Vert \blank \Vert)} \diff \epsilon.
  \end{equation*}
  Lemma 3.4.2 in \cite{van1996weak} provides the bound
  \begin{equation}
    \label{eq:12}
    \E^*_P{\left[ \Gamma_n(\delta)  \right]}
    \lesssim
    J_{[\,]}(\delta, \mathcal{L}_M, \Vert \blank \Vert_P)
    \left(
      1 + \frac{J_{[\,]}(\delta, \mathcal{L}_M, \Vert \blank \Vert_P)}{\delta^2 \sqrt{n}}C
    \right).
  \end{equation}
  \cite{bibaut2019fast} established that
  \begin{equation*}
    \log N_{[\,]}(\epsilon, \D_M, \Vert \blank \Vert_{\leb}) \lesssim
    \epsilon^{-1}\vert\log(\epsilon/M)\vert^{2(d-1)},
  \end{equation*}
  for $\epsilon \in (0,1)$, and so we have by assumption
  \begin{equation}
    \label{eq:28}
    \log{N_{[\,]}(\epsilon, \mathcal{L}_M, \| \blank \|_{P})} \leq \log{C}
    +
    \kappa\log{N_{[\,]}(\epsilon/C, \D_M, \| \blank \|_{\leb} )}
    \lesssim
    \epsilon^{-1}|\log
    {\left(
        \epsilon
      \right)}|^{2(d-1)},
  \end{equation}
  for small enough $\epsilon$. Using integration by parts we have
  \begin{align*}
    \int_0^{\delta} \sqrt{\epsilon^{-1}
    |\log
    { 
    \epsilon
    }|
    ^{2(d-1)}} \diff \epsilon
    & = (-1)^{d-1}
      \int_0^{\delta} \epsilon^{-1/2}
      {\left( \log
      \epsilon
      \right)}^{d-1} \diff \epsilon
    \\
    & =  (-1)^{d-1}
      \left(
      \delta^{1/2}
      {\left( \log
      \delta
      \right)}^{d-1} -     
      (d-1)\int_0^{\delta} \epsilon^{1/2}
      {\left( \log
      \epsilon
      \right)}^{d-2} \epsilon^{-1} \diff \epsilon
      \right)
    \\
    & =  
      \delta^{1/2}
      |\log
      \delta|
      ^{d-1} +     
      (d-1)\int_0^{\delta} \epsilon^{-1/2}
      |\log
      \epsilon |^{d-2} \diff \epsilon.
  \end{align*}
  As the second term on the right vanishes for $\delta \rightarrow 0$, we can
  use this and equation~(\ref{eq:28}) to obtain
  \begin{equation*}
    J_{[\,]}(\delta, \mathcal{L}_M, \Vert \blank \Vert_P)
    \lesssim
    \delta^{1/2}
    |\log
    \delta|
    ^{d-1},
  \end{equation*}
  and so equation~(\ref{eq:12}) gives
  \begin{equation*}
    \E^*_P{\left[ \Gamma_n(\delta)  \right]}
    \lesssim
    \delta^{1/2}
    |\log
    \delta|
    ^{d-1}
    \left(
      1 + \frac{\delta^{1/2}
        |\log
        \delta|
        ^{d-1}}{\delta^2 \sqrt{n}} M
    \right)
    \lesssim
    \delta^{1/2}
    |\log
    \delta|
    ^{d-1}
    +
    \frac{|\log
      \delta|
      ^{2(d-1)}}{\delta \sqrt{n}},
  \end{equation*}
  which was the first statement of the lemma. For the second statement, set
  $\delta=r_n^{-1}$ and obtain for all $n \geq 3$,
  \begin{equation*}
    \begin{split}
      n^{-1/2}r_n^2\E\left[ \Gamma_n(r_n^{-1})
      \right]
      & \lesssim
        n^{-1/2} r_n^2
        \left(
        r_n^{-1/2}
        |\log(r_n)|^{d-1}
        + \frac{r_n |\log(r_n)|^{2(d-1)}}{ \sqrt{n}}
        \right)      
      \\
      & \leq
        n^{-1/2}
        r_n^2
        \left(
        r_n^{-1/2}
        |\log(n)|^{d-1}
        + \frac{r_n |\log(n)|^{2(d-1)}}{ \sqrt{n}}
        \right)
      \\
      & =
        n^{-1/2}
        r_n^2
        \left(
        n^{-1/6} |\log(n)|^{4(d-1)/3}
        + \frac{n^{1/3} |\log(n)|^{4(d-1)/3}}{ \sqrt{n}}
        \right)
      \\
      & =
        n^{-1/2} r_n^2
        \left(
        n^{-1/6} |\log(n)|^{4(d-1)/3}
        + n^{-1/6} |\log(n)|^{4(d-1)/3}
        \right)
      \\
      & =
        n^{1/6} |\log(n)|^{-4(d-1)/3}
        2
        n^{-1/6} |\log(n)|^{4(d-1)/3}
      \\
      & =
        2
    \end{split}
  \end{equation*}
\end{proof}

\begin{proof}[Proof of Theorem~\ref{theorem:main-result}]
  We apply theorem 3.4.1 from \cite{van1996weak} to a HAL estimator
  \( \hat{f}_n \). This yields that \( \hat{f}_n \) converges to
  \( \hat{\pi}_n(\f) \) at rate \( r_n \) if there exist numbers
  $0 \leq \delta_n < \eta$ such that the following conditions hold for all
  $n \in \N$ and $\delta \in (\delta_n, \eta)$. %
  \newcommand{\C}{\text{(C1)}} %
  \newcommand{\CC}{\text{(C2)}} %
  \newcommand{\CCC}{\text{(C3)}} %
  \newcommand{\CCCC}{\text{(C4)}} %
  \begin{enumerate}
  \item[\C] Define for \( 0<a<b \) the hollow sphere
    \( B_{(a, b)}(f_0) = \{ f \in \D_M : a < \| f - f_0 \| < b\} \). It holds
    that
    \begin{equation*}
      \inf_{f \in B_{(\delta/2, \delta)}(\hat{\pi}_n(\f))} P{[ L(f, \blank) - L(\hat{\pi}_n(f), \blank)]} \gtrsim \delta^2.
    \end{equation*}
  \item[\CC] There exists a function
    $\phi_n\colon(\delta_n, \eta) \rightarrow \R$ such that
    \(\delta \mapsto \phi_n(\delta)/\delta^{\alpha}\) is decreasing for some
    $\alpha < 2$ and
    \begin{equation*}
      \E^*{\left[ \Gamma_n(\delta) \right]} \lesssim \phi_n(\delta) 
      \quad \text{and} \quad 
      n^{-1/2}\phi_n(r_n^{-1}) \leq r_n^{-2},
    \end{equation*}
    where \( \E^* \) denotes outer expectation.
  \item[\CCC]
    \( \empmeas{[L(\hat{f}_n, \blank)]} \leq \empmeas{[L(\hat{\pi}_n(\f),
      \blank)]} + \bigO_P(r_n^{-2})\).
  \item[\CCCC]
    \( \Vert \hat{f}_n - \hat{\pi}_n(\f) \Vert_{\leb} \arrow{P^*} 0 \), where
    \( P^* \) denotes outer probability.
  \end{enumerate}
  By Assumption~\ref{assum:loss-assum}~\ref{item:2} there exists
  \( C \in (1, \infty) \) such that
  \begin{equation*}
    \frac{1}{C}\| f - \f \|^2 \leq P{[L(f, \blank) - L(\f, \blank)]}
    \leq C \| f - \f \|^2.
  \end{equation*}
  Set $\delta_n = 4C^2\|\hat{\pi}_n(\f) -\f\|$. If
  $f \in B_{(\delta/2, \delta)}(\hat{\pi}_n(\f))$ we have that
  \begin{equation*}
    \delta/2 < \| f - \hat{\pi}_n(\f) \|
    \leq \| f - \f \| + \| \f -
    \hat{\pi}_n(\f) \|
    =  \| f - \f \| +
    \frac{\delta_n}{4 C^2}
    <   \| f - \f \| +
    \delta/4,
  \end{equation*}
  so
  \begin{equation}
    \label{eq:52}
    \| f - \f \| > \delta/4.
  \end{equation}
  Equation~(\ref{eq:52}) shows that
  \( B_{(\delta/2, \delta)}(\hat{\pi}_n(\f)) \subset B_{(\delta/4, \infty)}(\f)
  \), so
  \begin{equation}
    \label{eq:53}
    \inf_{f \in B_{(\delta/2, \delta)}(\hat{\pi}_n(\f))} P{[L(f, \blank) - L(\hat{\pi}_n(f), \blank)]} \geq
    \inf_{f \in B_{(\delta/4, \infty)}(\f)} P{[      
      L(f, \blank)
      - L(\hat{\pi}_n(f), \blank)]}.
  \end{equation}
  As
  \begin{equation*}
    P{[L(\hat{\pi}_n(f), \blank)- L(\f, \blank)]} \leq C \|\f - \hat{\pi}_n(f)
    \|^2
    = \frac{ C \delta_n^2}{16 C^4}
    = \frac{\delta_n^2}{16 C^3}
    < \frac{\delta^2}{16 C^3},
  \end{equation*}
  we have
  \begin{align*}
    P{[L(f, \blank)- L(\hat{\pi}_n(f), \blank)]}
    & = P{[L(f, \blank)- L(\f, \blank)]}
      - P{[L(\hat{\pi}_n(f), \blank)- L(\f, \blank)]}.
    \\
    &  > P{[L(\f, \blank) - L(f, \blank)]} - \frac{\delta^2}{16 C^3}
    \\
    &  > \frac{1}{C}\| \f - f \|^2 - \frac{\delta^2}{16 C^3},
  \end{align*}
  so when \( f \in B_{(\delta/4, \infty)}(\f) \), we have
  \begin{equation*}
    P{[L(f, \blank)- L(\hat{\pi}_n(f), \blank)]} >
    \frac{1}{C} \frac{\delta^2}{16} - \frac{\delta^2}{16 C^3}
    >
    \delta^2 \frac{1}{16 C}
    \left(
      1 - \frac{1}{C^3}
    \right)
    \gtrsim \delta^2.
  \end{equation*}
  Together with equation~(\ref{eq:53}) this shows that condition~\C{} holds.
  Because of Assumptions~\ref{assum:loss-assum}~\ref{item:1} and
  Assumptions~\ref{assum:loss-assum}~\ref{item:5}, condition~\CC{} holds
  by Lemma~\ref{lemma:bound-emp-brack}. By definition of \( \hat{f}_n \) and
  \( \hat{\pi}_n(\f) \), we have
  \begin{equation}
    \label{eq:51}
    \empmeas{[L(\hat{f}_n, \blank)]} \leq \empmeas{[L(\hat{\pi}_n(\f), \blank)]},
  \end{equation}
  so condition~\CCC{} is trivially true. We now show that
  condition~\CCCC{} holds. By
  Assumption~\ref{assum:loss-assum}~\ref{item:2} and equation~(\ref{eq:51}) we
  can write
  \begin{align*}
    \frac{1}{C}\Vert \f - \hat{f}_n \Vert_{\leb}^2 
    & \leq P{[L(\hat{f}_n, \blank) - L(\f,
      \blank)]}
    \\
    & =
      (P- \empmeas){[L(\hat{f}_n, \blank) - L(\f, \blank)]}
      + \empmeas{[L(\hat{f}_n, \blank) - L(\f, \blank)]} 
    \\
    & \leq
      (P- \empmeas){[L(\hat{f}_n, \blank) - L(\f, \blank)]}
      + \empmeas{[L(\hat{\pi}_n(\f), \blank) - L(\f, \blank)]}
    \\
    & =
      (P- \empmeas){[L(\hat{f}_n, \blank) - L(\f, \blank)]}
      + P{[L(\hat{\pi}_n(\f), \blank) - L(\f, \blank)]}
    \\
    & \quad
      - (P-\empmeas){[L(\hat{\pi}_n(\f), \blank) - L(\f, \blank)]} 
    \\
    & \leq
      4 \sup_{L\in \mathcal{L}_M}\left| (\empmeas-P){[L]} \right|
      + 
      P{[L(\hat{\pi}_n(\f), \blank) - L(\f, \blank)]}    .
  \end{align*}
  By Assumption~\ref{assum:loss-assum}~\ref{item:2} and
  Lemma~\ref{lemma:projection}, the second term on the right hand side converges
  to zero in probability. Proposition~1 in \citep{bibaut2019fast} and
  Theorem~2.4.1 in \citep{van1996weak} together with
  Assumption~\ref{assum:loss-assum}~\ref{item:5} imply that \( \mathcal{L}_M \)
  is a Glivenko-Cantelli class of functions. This implies that also the first
  term on the right converges to zero in probability, so
  \( \Vert \f - \hat{f}_n \Vert_{\leb} \arrow{P^*} 0 \). By
  Lemma~\ref{lemma:projection}, we also have
  \( \Vert \f - \hat{\pi}_n(\f) \Vert_{\leb} \arrow{P^*} 0 \), so this implies
  condition~\CCCC{}.
\end{proof}

\section{Additional proofs}
\label{sec:additional-proofs}

\subsection{Right-censored data}
\label{sec:negat-log-likel}

\begin{proof}[Proof of Proposition~\ref{prop:emp-risk-min-surv-dont-exist}]
  Let \( f^{\circ} \in \DD{1}_M \) be a function and \( j \in \{1, \dots, n-1\} \) an
  index such that \( f^{\circ}(\tilde{T}_{(j)}) > f^{\circ}(\tilde{T}_{(j+1)}) \). We shall
  construct a function \( \check{f} \in\DD{1}_M\) such that
  \( \empmeas{[L^{\mathrm{pl}}(\check{f}, \blank)]} < \empmeas{[L^{\mathrm{pl}}(f^{\circ}, \blank)]}\) when \( L^{\mathrm{pl}} \)
  is the negative log-likelihood defined in
  equation~(\ref{eq:nll-loss-survival}). This implies that \( f^{\circ} \) cannot be the
  minimizer of the empirical risk over \( \DD{1}_M \). To find \( \check{f} \)
  we first define
  \begin{equation*}
    V = \inf_{u \in [\tilde{T}_{(j)}, \tilde{T}_{(j+1)}]} f^{\circ}(u),
  \end{equation*}
  and
  \begin{equation*}
    f_{\epsilon}(t) = \1{\{t \in [\tilde{T}_{(j)} + \epsilon,
      \tilde{T}_{(j+1)})\}}V +
    \1{\{t \not\in [\tilde{T}_{(j)} + \epsilon, \tilde{T}_{(j+1)})\}} f^{\circ}(t),
    \quad
  \end{equation*}
  for \( \epsilon \in ( 0, [\tilde{T}_{(j+1)} - \tilde{T}_{(j)}]/2 ) \). In
  words, \( f_{\epsilon} \) is identical to \( f^{\circ} \), except on the
  interval \( [\tilde{T}_{(j)} + \epsilon, \tilde{T}_{(j+1)}) \) where it is
  constant and equals \( V \). Note that by the assumption that
  \( f^{\circ}(\tilde{T}_{(j)}) > f^{\circ}(\tilde{T}_{(j+1)}) \) we must have
  \begin{equation}
    \label{eq:6}
    f^{\circ}(\tilde{T}_{(j)}) > V. 
  \end{equation}
  As \( f^{\circ} \) is càdlàg, so is \( f_{\epsilon} \), and as \( f_{\epsilon} \) does
  not fluctuate more than \( f^{\circ} \), we must have
  \( \Vert f_{\epsilon} \Vert_v \leq \Vert f^{\circ} \Vert_v\). Thus
  \( f_{\epsilon} \in \DD{1}_M \). Now, by equation~(\ref{eq:6}) and because
  \( f^{\circ} \) is continuous from the right, we can find a $\delta>0$ and an
  $\epsilon_0 >0$ such that \( f^{\circ}(t) > V + \delta \) for all
  \( t \in [\tilde{T}_{(j)}, \tilde{T}_{(j)} + \epsilon_0] \). Thus, if we
  define \( \check{f} = f_{\epsilon_0/2} \) and
  $\mathcal{I} = (\tilde{T}_{(j)} + \epsilon_0/2, \tilde{T}_{(j)} +
  \epsilon_0)$, this implies that \( \check{f}(t) \leq f^{\circ}(t) \) for all
  $t \in [0,1] $ and \( \check{f}(t)< f^{\circ}(t)-\delta \) for
  \( t \in \mathcal{I} \). This in turn implies that
  \begin{equation}
    \label{eq:19}
    \int_0^{\tilde{T}_i} e^{\check{f}(u)} \diff u
    = \int_0^{\tilde{T}_i} e^{f^{\circ}(u)}  \diff u,
    \quad \text{for all} \quad
    i \leq j,
  \end{equation}
  and
  \begin{equation}
    \label{eq:7}
    \int_0^{\tilde{T}_i} e^{\check{f}(u)} \diff u
    < \int_0^{\tilde{T}_i} e^{f^{\circ}(u)}  \diff u,
    \quad \text{for all} \quad
    i > j.
  \end{equation}
  Finally, we have by construction that
  \begin{equation}
    \label{eq:1}
    \check{f}(\tilde{T}_i) = f^{\circ}(\tilde{T}_i)  \quad \text{for all} \quad   i \in \{1, \dots, n\} .
  \end{equation}
  Equations~(\ref{eq:19})-(\ref{eq:1}) together imply that
  \( \empmeas{[L^{\mathrm{pl}}(\check{f}, \blank)]} < \empmeas{[L^{\mathrm{pl}}(f^{\circ}, \blank)]}\).
\end{proof}

\begin{proof}[Proof of Proposition~\ref{prop:hal-surv-exists}]
  Let
  \(\mathcal{B}_M = \{\beta \in \R^{m(d,n)} : \Vert \beta \Vert_1 \leq M\}\). By
  construction, for any \(\mathbf{w}\) and \(\beta \in \mathcal{B}_M\) the map
  \(s\mapsto f_{\beta,n}(s,\mathbf{w})\) is constant on
  \([\tilde{T}_{(j-1)}, \tilde{T}_{(j)})\) for all \(j=1, \dots, n'\), and thus
  we can write
  \begin{equation}
    \label{eq:14}
    \int_0^{\tilde{T}_i}
    e^{f_{\beta,n}(s, W_i)} \diff s
    =     \sum_{j=1}^{n'} \1{\{\tilde{T}_i \geq \tilde{T}_{(j-1)}
      \}}(\tilde{T}_{(j)}  \wedge \tilde{T}_i-\tilde{T}_{(j-1)}) e^{f_{\beta,n}(\tilde{T}_{(j-1)}, W_i)}.
  \end{equation}
  For any \(t\) and \(\mathbf{w}\), the map \(\beta \mapsto f_{\beta,n}(t,\mathbf{w})\) is linear
  and as \( z \mapsto e^z \) is convex and non-decreasing it follows that
  \(\beta \mapsto e^{f_{\beta,n}(t,\mathbf{w})}\) is convex
  \citep[][Section~3.2.4]{boyd2004convex}. Thus equation~(\ref{eq:14}) implies
  that the map
  \( \beta \mapsto \empmeas{[L^{\mathrm{pl}}(f_{\beta,n}, \blank)]} \) is
  convex, and as $\mathcal{B}_M$ is convex it follows that the problem
  in~(\ref{eq:32}) is convex. The minimum is attained because the map
  \( \beta \mapsto \empmeas{[L^{\mathrm{pl}}(f_{\beta,n}, \blank)]} \) is
  continuous and $\mathcal{B}_M$ is compact.
\end{proof}

\begin{lemma}
  \label{lemma:haz-dens-equi}
  Let \( T \in [0, 1 + \epsilon] \) for some $\epsilon>0$,
  \( W \in [0,1]^{d-1} \), and assume that the conditional distribution of
  \( T \mid W=w \) has a Lebesgue density for all \( w \in [0,1]^{d-1} \). For
  $q$ a conditional density function for \( T \) given \( W \), let \( h_q \)
  the associated hazard function and \( S_q \) the associated survival function.
  Let
  \( \mathcal{Q}_M = \{q : \sup_{t \in [0,1]} \Vert h_q(t, \blank)
  \Vert_{\infty} \leq M \} \) for some \( M < \infty \). Let $\leb$ denote
  Lebesgue measure on \( [0,1]^d \). The following holds for all
  \( q, p \in \mathcal{Q}_M \).
  \begin{enumerate}[label=(\roman*)]
  \item
    \label{item:haz-dens-i}
    \( \left\Vert S_q- S_p \right\Vert_{\leb} \leq \left\Vert q- p
    \right\Vert_{\leb} \) and
    \( \left\Vert S_q- S_p \right\Vert_{\leb} \leq \left\Vert h_q- h_p
    \right\Vert_{\leb} \).
  \item \label{item:haz-dens-ii}
    \( \Vert h_p - h_q \Vert_{\leb} \leq (e^M + M e^{2M})\Vert p - q
    \Vert_{\leb} \) and
    \(\Vert p - q \Vert_{\leb} \leq (e^M + M) \Vert h_p - h_q \Vert_{\leb} \).
  \end{enumerate}
\end{lemma}
\begin{proof}[Proof of Lemma~\ref{lemma:haz-dens-equi}]
  The first inequality in statement~\ref{item:haz-dens-i} follows from Jensen's
  inequality:
  \begin{equation}
    \label{eq:49}
    \begin{split}
      \left\Vert 
      S_q- S_p
      \right\Vert_{\leb}^2
      & = \int_{[0,1]^{d-1}} \int_{[0,1]}  
        \left\{
        [S_q - S_p](s, \mathbf{w})
        \right\}^2 \diff s \diff \mathbf{w}
      \\
      & = \int_{[0,1]^{d-1}} \int_{[0,1]}  
        \left\{
        \int_0^s [p-q](u, \mathbf{w})  \diff u
        \right\}^2 \diff s \diff \mathbf{w}
      \\
      & = \int_{[0,1]^{d-1}} \int_{[0,1]}  
        s^2 \left\{
        \int_0^s [p-q](u, \mathbf{w})  \frac{\diff u}{s}
        \right\}^2 \diff s \diff \mathbf{w}
      \\
      & \leq \int_{[0,1]^{d-1}} \int_{[0,1]}  
        s^2 \int_0^s \left\{
        [p-q](u, \mathbf{w})  
        \right\}^2 \frac{\diff u}{s} \diff s \diff \mathbf{w}
      \\
      & \leq \int_{[0,1]^{d-1}} \int_{[0,1]}  
        \int_0^s \left\{
        [p-q](u, \mathbf{w})  
        \right\}^2 \diff u \diff s \diff \mathbf{w}
      \\
      & \leq \int_{[0,1]^{d-1}} \int_{[0,1]}  
        \int_0^1 \left\{
        [p-q](u, \mathbf{w})  
        \right\}^2 \diff u \diff s \diff \mathbf{w}
      \\
      & = \Vert q- p \Vert_{\leb}^2.
    \end{split}
\end{equation}
  For the second inequality in statement~\ref{item:haz-dens-i}, let \( H_q \)
  denote the conditional cumulative hazard function associated with the density
  \( q \). By the mean value theorem we may write
  \begin{equation*}
    e^{-H_q} -  e^{-H_p} = e^{-H_{q, p}}(H_p - H_q),
  \end{equation*}
  for some positive function \( H_{q, p} \). Hence by Hölder's inequality
  \begin{equation*}
    \left\Vert 
      S_q- S_p
    \right\Vert_{\leb}^2
    \leq
    \left\Vert 
      H_p -  H_q
    \right\Vert_{\leb}^2
    = \int_{[0,1]^{d-1}} \int_{[0,1]}  
    \left\{
      \int_0^s [h_p-h_q](u, \mathbf{w})  \diff u
    \right\}^2 \diff s \diff \mathbf{w},
  \end{equation*}
  and so Jensen's inequality (in the same way as in equation~(\ref{eq:49}))
  gives that
  \begin{equation*}
    \left\Vert 
      S_q- S_p
    \right\Vert_{\leb}^2 \leq
    \left\Vert 
      h_q- h_p
    \right\Vert_{\leb}^2.
  \end{equation*}
  This shows statement~\ref{item:haz-dens-i}. To obtain the first inequality in
  statement~\ref{item:haz-dens-ii} we write
  \begin{align*}
    \Vert h_q - h_p \Vert_{\leb} = 
    \left\Vert
    \frac{q}{S_q} - \frac{p}{S_p}
    \right\Vert_{\leb}
    & \leq  \left\Vert
      \frac{1}{S_q}(q-p)
      \right\Vert_{\leb}
      + \left\Vert p
      \left(
      \frac{1}{S_q}- \frac{1}{S_p}
      \right)
      \right\Vert_{\leb}
    \\
    & \leq \Vert S_q^{-1} \Vert_{\infty}  \left\Vert
      q-p
      \right\Vert_{\leb}
      + \Vert p \Vert_{\infty}
      \left\Vert 
      \left(
      \frac{1}{S_q}- \frac{1}{S_p}
      \right)
      \right\Vert_{\leb},
  \end{align*}
  where we use \( \Vert \blank \Vert_{\infty} \) to denote the supremum norm on
  \( [0,1]^d \). By the mean value theorem we may write
  \begin{equation*}
    \frac{1}{S_q}- \frac{1}{S_p} = \frac{1}{S_{q,p}^2}(S_q- S_p),
  \end{equation*}
  for a function \( S_{q,p} \) such that
  \( S_q \wedge S_p < S_{q,p} < S_q \vee S_p \). It follows that
  \begin{equation}
    \label{eq:45}
    \begin{split}
      \Vert h_q - h_p \Vert_{\leb} 
      & \leq \Vert S_q^{-1} \Vert_{\infty}  \left\Vert
        q-p
        \right\Vert_{\leb}
        + \Vert p \Vert_{\infty}
        \Vert (S_q \wedge S_p)^{-2} \Vert_{\infty}
        \left\Vert 
        \left(
        S_q- S_p
        \right)
        \right\Vert_{\leb}
      \\
      &\leq
        \left(
        \Vert S_q^{-1} \Vert_{\infty}
        +
        \Vert p \Vert_{\infty}
        \Vert (S_q \wedge S_p)^{-2} \Vert_{\infty}
        \right)
        \left\Vert
        q-p
        \right\Vert_{\leb},
    \end{split}
  \end{equation}
  where the last inequality follows from statement~\ref{item:haz-dens-i}. As
  \( p \in \mathcal{Q}_M \) and \( p = h_p S_p \) we have
  \( \Vert p \Vert_{\infty} \leq M \). As
  \( \Vert S_p^{-1} \Vert_{\infty} \leq \Vert \exp{\{\int_0^1 h_p(s, \blank)
    \diff s \}} \Vert_{\infty} \leq e^{M} \) and \( p,q \in \mathcal{Q}_M \),
  \( S_q^{-1} \), we have
  \begin{equation*}
    \left(
      \Vert S_q^{-1} \Vert_{\infty}
      +
      \Vert p \Vert_{\infty}
      \Vert (S_q \wedge S_p)^{-2} \Vert_{\infty}
    \right)
    \leq e^M + M e^{2M}.
  \end{equation*}
  This shows the first inequality in statement~\ref{item:haz-dens-ii}. For the
  second inequality we write
  \begin{align*}
    \Vert q - p \Vert_{\leb}
    &
      = \Vert (h_q - h_p)S_q \Vert_{\leb} +  \Vert h_p
      (S_q-S_p) \Vert_{\leb}
    \\
    & \leq
      \Vert S_q \Vert_{\infty}\Vert h_q - h_p \Vert_{\leb} +
      \Vert h_p
      \Vert_{\infty} \Vert  S_q-S_p \Vert_{\leb}
    \\
    & \leq
      \left(
      \Vert S_q\Vert_{\infty} +       \Vert h_p
      \Vert_{\infty} 
      \right)\Vert h_q - h_p \Vert_{\leb},
    \\
    & \leq
      \left(
      e^M + M 
      \right)\Vert h_q - h_p \Vert_{\leb},
  \end{align*}
  where the second to last inequality follows from
  statement~\ref{item:haz-dens-i}.
\end{proof}

\begin{proof}[Proof of Lemma~\ref{lemma:unique-min}]
  To show Lemma~\ref{lemma:unique-min} we need to find constants \( 0<c_M<C_M<
  \infty \) such that
  \begin{align}
    \label{eq:46}
    P[{L^{\mathrm{pl}}(f, \blank)-L^{\mathrm{pl}}(f_P, \blank)}]
    & \geq c_M \Vert f-f_P\Vert_{\leb}^2
          \intertext{and}
        \label{eq:47}
      P[{L^{\mathrm{pl}}(f, \blank)-L^{\mathrm{pl}}(f_P, \blank)}]
    & \leq C_M \Vert f-f_P\Vert_{\leb}^2.
  \end{align}
  Let \( P_{f} \) denote the distribution of the observed data induced by the
  marginal density \( P_W \), the conditional hazard for censoring
  \( \gamma_P \), and the conditional log-hazard for the event time of interest
  \( f \). Let
  \( \nu = P_W \otimes (\leb \otimes \tau + \delta_{\{1\} \times \{0\}}) \)
  denote a measure on the sample space
  \( \mathcal{O} = [0, 1]^{d-1} \times [0,1] \times \{0,1\} \) where \( \leb \)
  denotes Lebesgue measure, \( \tau \) the counting measure, \( \delta \) Dirac
  measure, and \( \leb \otimes \tau \) and \( \delta_{\{1\} \times \{0\}} \) are
  considered as measures on \( [0,1] \times \{0,1\} \). Then for every
  \( f \in \D_M\), \( P_f \ll \nu \) and if we let \( p_f \) denote the
  Radon-Nikodym derivative of \( P_f \) with respect to \( \nu \) we have a.s.,
\begin{equation}
  \label{eq:21}
  \begin{split}
    p_f(\mathbf{w}, t, \delta)
    & = 
      \left(
      e^{f(t, \mathbf{w})}\exp{
      \left\{
      - \int_0^t [e^{f(s, \mathbf{w})} + \gamma_P(s,\mathbf{w}) ]  \diff s
      \right\}}
      \right)^{\delta}
    \\
    & \quad\times
      \left(
      \gamma_P(t, \mathbf{w})^{\1_{[0,1)}(t)}
      \exp
      {\left\{
      {- \int_0^t [e^{f(s, \mathbf{w})} + \gamma_P(s,\mathbf{w}) ]  \diff s}
      \right\}}
      \right)^{1-\delta}
    \\
    &= 
      \left(
      e^{f(t, \mathbf{w})}
      \right)^{\delta}
      \exp{
      \left\{
      - \int_0^t e^{f(s, \mathbf{w})}    \diff s
      \right\}}
    \\
    & \quad\times
      \left(
      \gamma_P(t, \mathbf{w})^{\1_{[0,1)}(t)}      
      \right)^{1-\delta}
      \exp{
      \left\{
      - \int_0^t \gamma_P(s,\mathbf{w})    \diff s
      \right\}},
    \\
    & =: q_f(\mathbf{w}, t, \delta) g(\mathbf{w}, t, \delta),
  \end{split}
\end{equation}
where \( q_f \) denotes a component of the likelihood that depends only on
\( f \), and \( g \) denotes a component that depends only on \( \gamma_P \).
From this it follows that
\begin{equation}
  \label{eq:23}
  \begin{split}  
    \KL(P_{f_0} \, || \, P_f)
    & = \int \log{\frac{p_{f_0}}{p_f}} p_{f_0} \diff \nu
    \\
    & = \int_{[0,1]^d\times\{0,1\}}
      \bigg[
      \int_0^t e^{f(s,\mathbf{w})} \diff s - \delta f(t,\mathbf{w})
    \\  
    & \qquad -
      \left(
      \int_0^t e^{f_0(s,\mathbf{w})} \diff s - \delta f_0(t,\mathbf{w})
      \right)
      \bigg]
      p_{f_0}(\mathbf{w}, t, \delta)  \diff \nu(\mathbf{w}, t, \delta)
    \\
    & = P_{f_0}{[L^{\mathrm{pl}}(f, \blank)]} - P_{f_0}{[L^{\mathrm{pl}}(f_0, \blank)]},
  \end{split}
\end{equation}
where \( \KL\) is the Kullback-Leiber divergence. Following
\cite[p.~62]{van2000asymptotic} we have
\begin{equation}
  \label{eq:42}
  \begin{split}
    \KL(P_{f_0} \, || \, P_f)
    & \geq \int 
      \left(
      \sqrt{p_{f_0}} - \sqrt{p_{f}}
      \right)^2 \diff \nu
    \\
    & \geq
      \left(
      \Vert(\sqrt{p_{f_0}} + \sqrt{p_{f}})^2 \Vert_{\infty}
      \right)^{-1}
      \int 
      \left(
      p_{f_0} - p_{f}
      \right)^2 \diff \nu
    \\
    & \geq
      \left(      
      4 e^{M} (\| \gamma_P \|_{\infty} \vee 1)
      \right)^{-1}
      \int 
      \left(
      p_{f_0} - p_{f}
      \right)^2 \diff \nu
    \\
    & =
      \left(      
      4 e^{M} (\| \gamma_P \|_{\infty} \vee 1)
      \right)^{-1}
      \int 
      \left(
      p_{f_0} - p_{f}
      \right)^2 \diff \nu.
  \end{split}
\end{equation}
Let \( S_f(t, \mathbf{w}) = \exp(-\int_0^t e^{f(s, \mathbf{w})}\diff s) \)
denote the conditional survival function associated with the conditional hazard
function \( e^f \), and \( q_f^* = e^f S_f \) the conditional density associated
with the conditional hazard function \( e^f \). We have
\begin{equation}
  \label{eq:43}
  \begin{split}
    \int 
    \left(
    p_{f_0} - p_{f}
    \right)^2 \diff \nu
    & = \int 
      g^2\left(
      q_{f_0} - q_{f}
      \right)^2 \diff \nu
    \\
    & 
      \geq
      \int_{[0,1]^{d-1} \times [0,1) \times \{1\}}
      g^2\left(
      q_{f_0} - q_{f}
      \right)^2  \diff \nu
    \\
    & \geq e^{\Vert \gamma_P \Vert_{\infty}}\int_{[0,1]^{d-1}} \int_0^1  (q^*_{f_0} - q^*_{f})^2    
      \diff
      (\leb \otimes P_W)
    \\
    & 
      \geq
      e^{-\Vert \gamma_P \Vert_{\infty}}{\Vert \omega_P^{-1}\Vert_{\infty}}
      \Vert q^*_{f_0} - q^*_{f}\Vert_{\leb}^2.
  \end{split}
\end{equation}
By assumption,
\( e^{-\Vert \gamma_P \Vert_{\infty}}{\Vert \omega_P^{-1}\Vert_{\infty}} > 0 \)
and so Lemma~\ref{lemma:haz-dens-equi}~\ref{item:haz-dens-ii} and the mean value
theorem imply that
\begin{equation}
  \label{eq:48}
  \int 
  \left(
    p_{f_0} - p_{f}
  \right)^2 \diff \nu
  \geq \tilde{c}_M  \Vert f_0 - f\Vert_{\leb}^2,
\end{equation}
for some \( \tilde{c}_M \in (0, \infty) \). Combining
equations~(\ref{eq:23}),~(\ref{eq:42}) and~(\ref{eq:48}) gives
inequality~(\ref{eq:46}).

To show inequality~(\ref{eq:47}) we use, e.g.,
\cite[][Theorem~5]{gibbs2002choosing} to argue that
\begin{equation*}
  \KL(P_{f_0} \, || \, P_f) \leq
  \int \frac{(p_{f_0} - p_f)^2}{p_f}  \diff \nu.
\end{equation*}
Using the decomposition in equation~(\ref{eq:21}) we obtain
\begin{equation}
  \label{eq:24}
  \begin{split}
    \KL(P_{f_0} \, || \, P_f)
    & \leq
      \int \frac{ g^2(q_{f_0} - q_f)^2}{ q_f g}  \diff \nu
    \\
    & = \int \frac{ g (q_{f_0} - q_f)^2}{q_f}  \diff \nu
    \\
    & \leq
      \exp{\{M + e^{-M}\}} (\| \gamma_P \|_{\infty} \vee 1) \int  (q_{f_0} - q_f)^2  \diff \nu,
  \end{split}
\end{equation}
where we used that \( 1/q_{f} \) is bounded by \( \exp{\{M + e^{-M}\}} \) for
all \( f \in \D_M \), and that \( g \) is bounded by
\( \| \gamma_P \|_{\infty} \vee 1 \). Using that
\( [0,1]^{d-1}\times\{1\}\times\{1\} \) is a null set under $\nu$, we can write
\begin{equation}
  \label{eq:50}
  \begin{split}
    & \int  (q_{f_0} - q_f)^2  \diff \nu
    \\
    & =
      \int_{[0,1]^{d-1}\times [0, 1) \times \{1\}}   (q_{f_0} - q_f)^2
      \diff \nu    
      +
      \int_{[0,1]^{d-1}\times  [0,1] \times \{0\}}   (q_{f_0} - q_f)^2  \diff
      \nu
    \\
    & =
      \int_{[0,1]^{d-1}\times [0, 1)}   (q_{f_0}^* - q_f^*)^2
      \diff     
      (\leb \otimes P_w) +
      \int_{[0,1]^{d-1}\times  [0,1] }   (S_{f_0} - S_f)^2  \diff
      (\leb \otimes P_w)
    \\
    & \leq
      \int_{[0,1]^{d-1}\times [0, 1]}  
      \left\{
      (q_{f_0}^* - q_f^*)^2
      +  (S_{f_0} - S_f)^2  
      \right\}\diff
      (\leb \otimes P_w)
    \\
    & \leq \Vert\omega_P\Vert_{\infty}
      \left(
      \int_{[0,1]^{d}}   (q_{f_0}^* - q_f^*)^2
      + (S_{f_0} - S_f)^2  \diff
      \leb
      \right)
    \\
    & \leq \Vert\omega_P\Vert_{\infty}
      \left(
      \Vert q_{f_0}^* - q_f^* \Vert_{\leb}^2
      + \Vert S_{f_0}^* - S_f^* \Vert_{\leb}^2
      \right),
  \end{split}
\end{equation}
and the inequality~(\ref{eq:47}) then follows from
Lemma~\ref{lemma:haz-dens-equi} combined with
equations~(\ref{eq:23}),~(\ref{eq:24}) and~(\ref{eq:50}).
\end{proof}

\begin{proof}[Proof of Corollary~\ref{cor:main-survival}]
  First note that because condition~\ref{assum:f0}~\ref{item:11} is assumed to
  hold, \( f_P = \f \) a.e.\ by Lemma~\ref{lemma:unique-min}. Thus
  Corollary~\ref{cor:main-survival} follows from
  Theorem~\ref{theorem:main-result} if we can show that
  Assumption~\ref{assum:loss-assum} is true.
  Assumption~\ref{assum:loss-assum}~\ref{item:1} follows by definition of the
  loss function and \ref{assum:loss-assum}~\ref{item:2} follow from
  Lemma~\ref{lemma:unique-min} as $\gamma_P$ and \( \omega_P \) are assumed
  uniformly bounded. It thus only remains to show
  \ref{assum:loss-assum}~\ref{item:5}. To do so, let $\epsilon>0$ be given and
  let \([l_1, u_1], \dots, [l_K, u_K]\) denote a collection of
  \(\epsilon\)-brackets with respects to \( \| \blank \|_{\leb} \) covering
  $\D_M$. By definition of the bracketing number we can take
  \(K=N_{[\,]}(\epsilon, \D_M, \|\blank\|_{\leb})\). Define for all
  \(k=1, \dots, K\),
  \begin{align*}
    \tilde{l}_k(t, \delta, \mathbf{w})
    = \delta l_k(t, \mathbf{w})
    - \int_0^{t} e^{u_k(s, \mathbf{w})} \diff s,
    \quad \text{and} \quad 
    \tilde{u}_k(t, \delta, \mathbf{w})
    = \delta u_k(t, \mathbf{w})
    - \int_0^{t} e^{l_k(s, \mathbf{w})} \diff s.
  \end{align*}
  Any element in \( \mathcal{L}_M\) is on the form
  \( L^{\mathrm{pl}}(f, \blank)\) for some \(f \in \D_M \). If
  \([l_k, u_k]\) is a bracket containing \(f\) then it follows that
  \([\tilde{l}_k, \tilde{u}_k]\) contains \(L^{\mathrm{pl}}(f, \blank)\). Thus
  \([\tilde{l}_1, \tilde{u}_1], \dots, [\tilde{l}_K, \tilde{u}_K]\) is a
  collection of brackets covering $\mathcal{L}_M$. If we let $\E$ denote
  expectation under \( P \) we have by the triangle inequality
  \begin{align*}
    \Vert \tilde{l}_k - \tilde{u}_k \Vert_{P}
    & \leq \E{\left[ \Delta
      \left\{
      l_k(\tilde{T}, W) - u_k(\tilde{T}, W)
      \right\}^2 \right]}^{1/2}
    \\
    & \qquad
      +
      \E{\left[ 
      \left\{
      \int_0^{\tilde{T}}  e^{u_k(s, W)} -  e^{l_k(s, W)} \diff s          
      \right\}^2 \right]}^{1/2}.
  \end{align*}
  By equation~(\ref{eq:33}), \( \Delta = \Delta \1{\{\tilde{T} < 1\}} \) a.s.,
  which implies
  \begin{equation*}
    \Delta
    \left\{
      l_k(\tilde{T}, W) - u_k(\tilde{T}, W)
    \right\}^2
    \leq
    \1{\{\tilde{T} < 1\}}
    \left\{
      l_k(\tilde{T}, W) - u_k(\tilde{T}, W)
    \right\}^2 \quad \text{a.s.},
  \end{equation*}
  and so
  \begin{align*}
    & \E{\left[ \Delta
      \left\{
      l_k(\tilde{T}, W) - u_k(\tilde{T}, W)
      \right\}^2 \right]}
    \\
    & \leq
      \E{\left[ \1{\{\tilde{T} < 1\}}
      \left\{
      l_k(\tilde{T}, W) - u_k(\tilde{T}, W)
      \right\}^2 \right]}
    \\ 
    & = \int_{[0,1]^{d-1}} \int_0^1 \left\{
      l_k(s, \mathbf{w}) - u_k(s, \mathbf{w})
      \right\}^2 h(s,\mathbf{w}) e^{-\int_0^s h(u,\mathbf{w}) \diff u}\omega_P(\mathbf{w}) \diff s \diff \mathbf{w},
  \end{align*}
  where we use \( h(\blank,\mathbf{w}) \) to denote the conditional hazard for
  \( \tilde{T} \) on \( [0,1) \) given \( W=\mathbf{w} \). By assumption,
  \( \|h \omega_P\|_{\infty} \leq B \) for some finite constant \( B \), and so we
  obtain
  \begin{equation*}
    \E{\left[ \Delta
        \left\{
          l_k(\tilde{T}, W) - u_k(\tilde{T}, W)
        \right\}^2 \right]}^{1/2}
    \leq B \| l_k - u_k \|_{\leb}.
  \end{equation*}
  By Jensen's inequality and the mean value theorem we similarly obtain
  \begin{align*}
    \E{\left[ 
    \left\{
    \int_0^{\tilde{T}}  e^{u_k(s, W)} -  e^{l_k(s, W)} \diff s          
    \right\}^2 \right]}^{1/2}
    & \leq
      \E{\left[ {\tilde{T}}
      \int_0^{\tilde{T}}  
      \left(
      e^{u_k(s, W)} -  e^{l_k(s, W)}
      \right)^2 \diff s          
      \right]}^{1/2}
    \\
    & \leq
      \E{\left[ 
      \int_0^{1}  
      \left(
      e^{u_k(s, W)} -  e^{l_k(s, W)}
      \right)^2 \diff s          
      \right]}^{1/2}
    \\
    & \leq
      e^M \E{\left[ 
      \int_0^{1}  
      \left\{
      u_k(s, W) -  l_k(s, W)
      \right\}
      ^2 \diff s          
      \right]}^{1/2}
    \\
    & \leq e^M  \|\omega_P\|_{\infty} \|u_k - l_k \|_{\leb},
  \end{align*}
  and so we have
  \begin{equation*}
    \Vert \tilde{l}_k - \tilde{u}_k \Vert_{P}
    \leq
    \left(
      B + e^M  \|\omega_P\|_{\infty}
    \right) \|u_k - l_k \|_{\leb}.
  \end{equation*}
  Thus \([\tilde{l}_1, \tilde{u}_1], \dots, [\tilde{l}_K, \tilde{u}_K]\) is a
  collection of $(B + e^M \|\omega_P\|_{\infty})\epsilon$-brackets covering
  $\mathcal{L}_M$, which shows that
  \( N_{[\,]}(\epsilon, \mathcal{L}_M, \Vert \blank \Vert_{P}) \leq
  N_{[\,]}(\epsilon/(B + e^M \|\omega_P\|_{\infty}), \D_M, \Vert \blank
  \Vert_{\leb})\).
\end{proof}
\subsection{Density estimation}
\label{sec:density-estimation-1}

\begin{proof}[Proof of Proposition~\ref{prop:dens-convex}]
  Define
  \( \tilde{\mathcal{B}}_M = \{ \beta\in \R^{\tilde{m}(d,n)} : \| \beta \|_1
  \leq M\} \) where \( \tilde{m}(d,n) =n2^{d-1}\). To show that
  \( \beta \mapsto \empmeas{[\bar{L}(g_{\beta,n}, \blank)]} \) is convex, take
  \( \beta_1, \beta_0 \in \tilde{\mathcal{B}}_M \). Note that for any
  \( u \in [0,1] \), \( \mathbf{w} \in [0,1]^{d-1} \), and $\alpha\in[0,1]$,
  \begin{equation*}
    \exp
    {\left\{
        g_{\alpha \beta_1 + (1-\alpha)\beta_0}(z, \mathbf{w})
      \right\}}
    =  
    \left(
      \exp
      {\left\{
          g_{ \beta_1}(z, \mathbf{w})
        \right\}}
    \right)^\alpha
    \left(
      \exp
      {\left\{
          g_{ \beta_0}(z, \mathbf{w})
        \right\}}
    \right)^{1-\alpha}.
  \end{equation*}
  By Hölder's inequality,
  \begin{align*}
    \int_0^1 e^{g_{\alpha \beta_1 + (1-\alpha)\beta_0}(z, \mathbf{w})} \diff z
    & =
      \int_0^1     \left(
      \exp
      {\left\{
      g_{ \beta_1}(z, \mathbf{w})
      \right\}}
      \right)^\alpha
      \left(
      \exp
      {\left\{
      g_{ \beta_0}(z, \mathbf{w})
      \right\}}
      \right)^{1-\alpha} \diff z
    \\
    & \leq
      \left(
      \int_0^1
      \exp
      {\left\{
      g_{ \beta_1}(z, \mathbf{w})
      \right\}}
      \diff z\right)^\alpha      
      \left(
      \int_0^1
      \exp
      {\left\{
      g_{ \beta_0}(z, \mathbf{w})
      \right\}}
      \diff z
      \right)^{1-\alpha} ,
  \end{align*}
  which implies
  \begin{align*}
    \log\left(
    \int_0^1 e^{g_{\alpha \beta_1 + (1-\alpha)\beta_0}(s, \mathbf{w})} \diff s
    \right)
    & \leq
      \alpha
      \log\left(
      \int_0^1 e^{g_{ \beta_1}(s, \mathbf{w})} \diff s
      \right)
    \\
    & \qquad
      +(1-\alpha)
      \log\left(
      \int_0^1 e^{g_{ \beta_0}(s, \mathbf{w})} \diff s
      \right).
  \end{align*}
  From this it follows that
  \begin{equation*}
    \empmeas{[\bar{L}(g_{\alpha \beta_1 + (1-\alpha)\beta_0}, \blank)]}
    \leq \alpha \empmeas{[\bar{L}(g_{\beta_1 }, \blank)]}
    + (1-\alpha) \empmeas{[\bar{L}(g_{\beta_0 }, \blank)]}, 
  \end{equation*}
  so \( \beta \mapsto \empmeas{[\bar{L}(g_{\beta,n}, \blank)]} \) is convex.
  Because \(\tilde{\mathcal{B}}_M \) is convex the problem in (\ref{eq:20}) is
  convex, and because
  \( \beta \mapsto \empmeas{[\bar{L}(g_{\beta,n}, \blank)]} \) is continuous,
  the minimum is attained. To show the second statement in the proposition, note
  that
  \begin{equation}
    \label{eq:39}
    \empmeas{[-\log p]} = \empmeas{[\bar{L}(\log p, \blank)]}
    \quad \text{for any} \quad
    p \in \mathcal{P}_{M,n}^d.
  \end{equation}
  Observe that if \( a \colon [0,1]^d \rightarrow \R \) is a function such that
  \( a(u,\mathbf{w}) = a(0,\mathbf{w}) \) for all \( u \in [0,1] \) and
  \( \mathbf{w} \in [0,1]^{d-1} \), then for any \( f \in \D_M \) and
  \( O \in [0,1]^d \),
  \begin{equation}
    \label{eq:38}
    \begin{split}
      \bar{L}(f+ a, O)
      & =
        \log
        \left(
        \int_0^1 e^{f(s, W) + a(s,W)} \diff s
        \right)
        - (f(U, W) - a(U,W))
      \\
      & =
        \log
        \left(
        e^{a(0,W)} \int_0^1 e^{f(s, W) } \diff s
        \right)
        - (f(U, W) - a(0,W))
      \\
      & =     \bar{L}(f, O).
    \end{split}
  \end{equation}
  In particular, this holds when
  $a(\mathbf{x}) = b\1{\{X_{s, i} \preceq \mathbf{x}_{s}\}}$ for some
  \( b \in \R \), \( i \in \{1, \dots, n\} \), and \( s \not \in \mathcal{I} \).
  Hence by definition of \( \mathcal{P}_{M,n}^d \) we have for any
  \( p \in \mathcal{P}_{M,n}^d \) that
  \( \empmeas{[\bar{L}(\log p, \blank)]} = \empmeas{[\bar{L}(g_{\beta,n},
    \blank)]} \) for some \( \beta \in \tilde{\mathcal{B}}_M \). By
  equation~(\ref{eq:39}), we thus have that for any
  \( p \in \mathcal{P}_{M,n}^d \),
  \( \empmeas{[-\log p]} = \empmeas{[\bar{L}(g_{\beta,n}, \blank)]} \) for some
  \( \beta \in \tilde{\mathcal{B}}_M \). The result then follows from the
  definition of \( g_{\hat{\beta},n} \).
\end{proof}

\begin{proof}[Proof of Corollary~\ref{cor:hal-dens}]
  Define the log-density
  \begin{equation*}
    \f_1(u, \mathbf{w}) = \f(u, \mathbf{w}) - \log{(\int_0^1 e^{\f(z,
        \mathbf{w})} \diff z)},
  \end{equation*}
  and note that \( \f_1 \in \mathcal{P}_M^d \). Equations~(\ref{eq:39})
  and~(\ref{eq:38}) imply that \( P{[\bar{L}(\f, \blank)]} = P{[-\log \f_1]}\)
  and thus \( p_P = \f_1 \) a.e., because the log-likelihood is a strictly
  proper scoring rule \citep{gneiting2007strictly} and
  \( p_P \in \mathcal{P}_M^d \) by assumption. For any HAL estimator
  \( \hat{p}_n \) we can write
  \( \log \hat{p}_n(u, \mathbf{w}) = g_{\hat{\beta},n}(u, \mathbf{w}) -
  \log{(\int_0^1 e^{g_{\hat{\beta},n}(z, \mathbf{w})} \diff z)} \), for some
  solution \( \hat{\beta} \) to the problem (\ref{eq:20}). By
  equation~(\ref{eq:38}), \( g_{\hat{\beta},n} \) is a HAL estimator for the
  loss \( \bar{L} \) as defined in equation~(\ref{eq:estimator}). To prove
  Corollary~\ref{cor:hal-dens} it suffices to show that
  \begin{equation}
    \label{eq:40}
    \| g_{\hat{\beta},n} - \f \|_{\leb} = \smallO_P(n^{-1/3}\log(n)^{2(d-1)/3}).
  \end{equation}
  We show that Assumption~\ref{assum:loss-assum} holds for \( \bar{L} \), which
  imply that equation~(\ref{eq:40}) is true by
  Theorem~\ref{theorem:main-result}.
  Assumption~\ref{assum:loss-assum}~\ref{item:1} holds because all
  \( f \in \D_{M,n} \) are uniformly bounded, and
  Assumption~\ref{assum:loss-assum}~\ref{item:2} holds by properties of the
  Kullback-Leibler divergence because we assume that \( \omega_P \) is uniformly
  bounded away from zero and infinity \citep{gibbs2002choosing}.
  Assumption~\ref{assum:loss-assum}~\ref{item:5} is established by the same
  arguments used in the proof of Corollary~\ref{cor:main-survival}.
\end{proof}

\end{appendix}

\bibliography{bib.bib}

\begin{thebibliography}{50}
\providecommand{\natexlab}[1]{#1}
\providecommand{\url}[1]{\texttt{#1}}
\expandafter\ifx\csname urlstyle\endcsname\relax
  \providecommand{\doi}[1]{doi: #1}\else
  \providecommand{\doi}{doi: \begingroup \urlstyle{rm}\Url}\fi

\bibitem[Aistleitner and Dick(2015)]{ChristophAistleitner2015}
C.~Aistleitner and J.~Dick.
\newblock Functions of bounded variation, signed measures, and a general
  {K}oksma-{H}lawka inequality.
\newblock \emph{Acta Arithmetica}, 167\penalty0 (2):\penalty0 143--171, 2015.
\newblock URL \url{http://eudml.org/doc/279219}.

\bibitem[Andersen et~al.(2012)Andersen, Borgan, Gill, and
  Keiding]{andersen2012statistical}
P.~K. Andersen, O.~Borgan, R.~D. Gill, and N.~Keiding.
\newblock \emph{Statistical models based on counting processes}.
\newblock Springer Science \& Business Media, 2012.

\bibitem[Bibaut and van~der Laan(2019)]{bibaut2019fast}
A.~F. Bibaut and M.~J. van~der Laan.
\newblock Fast rates for empirical risk minimization over càdlàg functions
  with bounded sectional variation norm.
\newblock \emph{arXiv preprint arXiv:1907.09244}, 2019.

\bibitem[Bickel and Ritov(1988)]{bickel88:density}
P.~J. Bickel and Y.~Ritov.
\newblock Estimating integrated squared density derivates.
\newblock \emph{Sankhy$\bar{a}$ A}, 50:\penalty0 381--393, 1988.

\bibitem[Boyd and Vandenberghe(2004)]{boyd2004convex}
S.~P. Boyd and L.~Vandenberghe.
\newblock \emph{Convex optimization}.
\newblock Cambridge university press, 2004.

\bibitem[Chernozhukov et~al.(2018)Chernozhukov, Chetverikov, Demirer, Duflo,
  Hansen, Newey, and Robins]{chernozhukov2018double}
V.~Chernozhukov, D.~Chetverikov, M.~Demirer, E.~Duflo, C.~Hansen, W.~Newey, and
  J.~Robins.
\newblock Double/debiased machine learning for treatment and structural
  parameters, 2018.

\bibitem[Cox(1975)]{cox1975partial}
D.~R. Cox.
\newblock Partial likelihood.
\newblock \emph{Biometrika}, 62\penalty0 (2):\penalty0 269--276, 1975.

\bibitem[Coyle et~al.(2022)Coyle, Hejazi, Phillips, {van der Laan}, and {van
  der Laan}]{Coyle_Hejazi_Phillips_Laan_Laan_2022}
J.~R. Coyle, N.~S. Hejazi, R.~V. Phillips, L.~W. {van der Laan}, and M.~J. {van
  der Laan}.
\newblock \emph{{hal9001}: The scalable highly adaptive lasso}, 2022.
\newblock URL \url{https://github.com/tlverse/hal9001}.
\newblock R package version 0.4.3.

\bibitem[Czerebak-Morozowicz et~al.(2008)Czerebak-Morozowicz, Rychlik, and
  Urbanek]{czerebak2008almost}
E.~Czerebak-Morozowicz, Z.~Rychlik, and M.~Urbanek.
\newblock Almost sure functional central limit theorems for multiparameter
  stochastic processes.
\newblock \emph{Condensed Matter Physics}, 2008.

\bibitem[Dvoretzky et~al.(1956)Dvoretzky, Kiefer, and
  Wolfowitz]{dvoretzky1956asymptotic}
A.~Dvoretzky, J.~Kiefer, and J.~Wolfowitz.
\newblock Asymptotic minimax character of the sample distribution function and
  of the classical multinomial estimator.
\newblock \emph{The Annals of Mathematical Statistics}, pages 642--669, 1956.

\bibitem[Fang et~al.(2021)Fang, Guntuboyina, and Sen]{fang2021multivariate}
B.~Fang, A.~Guntuboyina, and B.~Sen.
\newblock Multivariate extensions of isotonic regression and total variation
  denoising via entire monotonicity and {H}ardy--{K}rause variation.
\newblock \emph{The Annals of Statistics}, 49\penalty0 (2):\penalty0 769--792,
  2021.

\bibitem[Ferger(2015)]{ferger2015arginf}
D.~Ferger.
\newblock Arginf-sets of multivariate cadlag processes and their convergence in
  hyperspace topologies.
\newblock \emph{Theory of Stochastic Processes}, 20\penalty0 (2):\penalty0
  13--41, 2015.

\bibitem[Friedman et~al.(2010)Friedman, Tibshirani, and
  Hastie]{Friedman_Tibshirani_Hastie_2010_Regularizat_Paths_Generalized}
J.~Friedman, R.~Tibshirani, and T.~Hastie.
\newblock Regularization paths for generalized linear models via coordinate
  descent.
\newblock \emph{Journal of Statistical Software}, 33\penalty0 (1):\penalty0
  1--22, 2010.
\newblock \doi{10.18637/jss.v033.i01}.

\bibitem[Fu et~al.(2020)Fu, Narasimhan, and Boyd]{Fu_Narasimhan_Boyd_2020_Cvxr}
A.~Fu, B.~Narasimhan, and S.~Boyd.
\newblock {CVXR}: An {R} package for disciplined convex optimization.
\newblock \emph{Journal of Statistical Software}, 94\penalty0 (14):\penalty0
  1--34, 2020.
\newblock \doi{10.18637/jss.v094.i14}.

\bibitem[Geman(1981)]{geman1981sieves}
S.~Geman.
\newblock Sieves for nonparametric estimation of densities and regressions.
\newblock \emph{Reports in Pattern Analysis}, 99, 1981.

\bibitem[Geman and Hwang(1982)]{geman1982nonparametric}
S.~Geman and C.-R. Hwang.
\newblock Nonparametric maximum likelihood estimation by the method of sieves.
\newblock \emph{The annals of Statistics}, pages 401--414, 1982.

\bibitem[Gibbs and Su(2002)]{gibbs2002choosing}
A.~L. Gibbs and F.~E. Su.
\newblock On choosing and bounding probability metrics.
\newblock \emph{International statistical review}, 70\penalty0 (3):\penalty0
  419--435, 2002.

\bibitem[Gill et~al.(1995)Gill, Laan, and Wellner]{gill1995inefficient}
R.~D. Gill, M.~J. Laan, and J.~A. Wellner.
\newblock Inefficient estimators of the bivariate survival function for three
  models.
\newblock In \emph{Annales de l'IHP Probabilit{\'e}s et statistiques},
  volume~31, pages 545--597, 1995.

\bibitem[Gneiting and Raftery(2007)]{gneiting2007strictly}
T.~Gneiting and A.~E. Raftery.
\newblock Strictly proper scoring rules, prediction, and estimation.
\newblock \emph{Journal of the American statistical Association}, 102\penalty0
  (477):\penalty0 359--378, 2007.

\bibitem[Goldstein and Khasminskii(1996)]{goldstein1996efficient}
L.~Goldstein and R.~Khasminskii.
\newblock On efficient estimation of smooth functionals.
\newblock \emph{Theory of Probability \& Its Applications}, 40\penalty0
  (1):\penalty0 151--156, 1996.

\bibitem[Goldstein and Messer(1992)]{Gold:Mess:92:OPE}
L.~Goldstein and K.~Messer.
\newblock {Optimal} plug-in estimators for nonparametric functional estimation.
\newblock \emph{Ann. Statist.}, 20:\penalty0 1306--1328, 1992.

\bibitem[Grenander(1981)]{grenander1981abstract}
U.~Grenander.
\newblock \emph{Abstract inference}.
\newblock Wiley, 1981.

\bibitem[Groeneboom and Jongbloed(2014)]{groeneboom2014nonparametric}
P.~Groeneboom and G.~Jongbloed.
\newblock \emph{Nonparametric estimation under shape constraints}.
\newblock Cambridge University Press, 2014.

\bibitem[Gu and Qiu(1993)]{gu1993smoothing}
C.~Gu and C.~Qiu.
\newblock Smoothing spline density estimation: Theory.
\newblock \emph{The Annals of Statistics}, 21\penalty0 (1):\penalty0 217--234,
  1993.

\bibitem[Hardy(1906)]{hardy1906double}
G.~H. Hardy.
\newblock On double {F}ourier series and especially those which represent the
  double zeta-function with real and incommensurable parameters.
\newblock \emph{Quart. J. Math}, 37\penalty0 (1):\penalty0 53--79, 1906.

\bibitem[Hejazi et~al.(2020)Hejazi, Coyle, and van~der Laan]{hejazi2020hal9001}
N.~S. Hejazi, J.~R. Coyle, and M.~J. van~der Laan.
\newblock hal9001: Scalable highly adaptive lasso regression inr.
\newblock \emph{Journal of Open Source Software}, 5\penalty0 (53):\penalty0
  2526, 2020.

\bibitem[Hothorn(2020)]{hothorn2020transformation}
T.~Hothorn.
\newblock Transformation boosting machines.
\newblock \emph{Statistics and Computing}, 30\penalty0 (1):\penalty0 141--152,
  2020.

\bibitem[Krause(1903)]{krause1903mittelwertsatze}
M.~Krause.
\newblock {\"U}ber {M}ittelwerts{\"a}tze im {G}ebiete der {D}oppelsummen and
  {D}oppelintegrale.
\newblock \emph{Leipziger Ber}, 55:\penalty0 239--263, 1903.

\bibitem[Lee et~al.(2021)Lee, Chen, and Ishwaran]{lee2021boosted}
D.~K. Lee, N.~Chen, and H.~Ishwaran.
\newblock Boosted nonparametric hazards with time-dependent covariates.
\newblock \emph{Annals of statistics}, 49\penalty0 (4):\penalty0 2101, 2021.

\bibitem[Leonard(1978)]{leonard1978density}
T.~Leonard.
\newblock Density estimation, stochastic processes and prior information.
\newblock \emph{Journal of the Royal Statistical Society: Series B
  (Methodological)}, 40\penalty0 (2):\penalty0 113--132, 1978.

\bibitem[McKeague and Utikal(1990)]{mckeague1990inference}
I.~W. McKeague and K.~J. Utikal.
\newblock Inference for a nonlinear counting process regression model.
\newblock \emph{The Annals of Statistics}, 18\penalty0 (3):\penalty0
  1172--1187, 1990.

\bibitem[Naaman(2021)]{naaman2021tight}
M.~Naaman.
\newblock On the tight constant in the multivariate
  {D}voretzky--{K}iefer--{W}olfowitz inequality.
\newblock \emph{Statistics \& Probability Letters}, 173:\penalty0 109088, 2021.

\bibitem[Neuhaus(1971)]{neuhaus1971weak}
G.~Neuhaus.
\newblock On weak convergence of stochastic processes with multidimensional
  time parameter.
\newblock \emph{The Annals of Mathematical Statistics}, 42\penalty0
  (4):\penalty0 1285--1295, 1971.

\bibitem[Owen(2005)]{owen2005multidimensional}
A.~B. Owen.
\newblock Multidimensional variation for quasi-monte carlo.
\newblock In \emph{Contemporary Multivariate Analysis And Design Of
  Experiments: In Celebration of Professor Kai-Tai Fang's 65th Birthday}, pages
  49--74. World Scientific, 2005.

\bibitem[Ramlau-Hansen(1983)]{ramlau1983smoothing}
H.~Ramlau-Hansen.
\newblock Smoothing counting process intensities by means of kernel functions.
\newblock \emph{The Annals of Statistics}, pages 453--466, 1983.

\bibitem[Rytgaard et~al.(2022)Rytgaard, Gerds, and van~der
  Laan]{rytgaard2022continuous}
H.~C. Rytgaard, T.~A. Gerds, and M.~J. van~der Laan.
\newblock Continuous-time targeted minimum loss-based estimation of
  intervention-specific mean outcomes.
\newblock \emph{The Annals of Statistics}, 50\penalty0 (5):\penalty0
  2469--2491, 2022.

\bibitem[Rytgaard et~al.(2023)Rytgaard, Eriksson, and van~der
  Laan]{rytgaard2021estimation}
H.~C. Rytgaard, F.~Eriksson, and M.~J. van~der Laan.
\newblock Estimation of time-specific intervention effects on continuously
  distributed time-to-event outcomes by targeted maximum likelihood estimation.
\newblock \emph{Biometrics}, 79\penalty0 (4):\penalty0 3038--3049, 2023.

\bibitem[Schmid and Hothorn(2008)]{schmid2008flexible}
M.~Schmid and T.~Hothorn.
\newblock Flexible boosting of accelerated failure time models.
\newblock \emph{BMC bioinformatics}, 9:\penalty0 1--13, 2008.

\bibitem[Schuler et~al.(2023)Schuler, Li, and van~der
  Laan]{schuler2023selectively}
A.~Schuler, Y.~Li, and M.~van~der Laan.
\newblock The selectively adaptive lasso.
\newblock \emph{arXiv preprint arXiv:2205.10697}, 2023.

\bibitem[Silverman(1982)]{silverman1982estimation}
B.~W. Silverman.
\newblock On the estimation of a probability density function by the maximum
  penalized likelihood method.
\newblock \emph{The Annals of Statistics}, pages 795--810, 1982.

\bibitem[Spierdijk(2008)]{spierdijk2008nonparametric}
L.~Spierdijk.
\newblock Nonparametric conditional hazard rate estimation: a local linear
  approach.
\newblock \emph{Computational Statistics \& Data Analysis}, 52\penalty0
  (5):\penalty0 2419--2434, 2008.

\bibitem[Stone(1980)]{Ston:80:ORC}
C.~J. Stone.
\newblock {Optimal} rates of convergence for nonparametric estimators.
\newblock \emph{Ann. Statist.}, 8:\penalty0 1348--1360, 1980.

\bibitem[Tay et~al.(2023)Tay, Narasimhan, and
  Hastie]{Tay_Narasimhan_Hastie_2023_Elastic_Net_Regularizat}
J.~K. Tay, B.~Narasimhan, and T.~Hastie.
\newblock Elastic net regularization paths for all generalized linear models.
\newblock \emph{Journal of Statistical Software}, 106\penalty0 (1):\penalty0
  1--31, 2023.
\newblock \doi{10.18637/jss.v106.i01}.

\bibitem[van~der Laan(2017)]{van2017generally}
M.~van~der Laan.
\newblock A generally efficient targeted minimum loss based estimator based on
  the highly adaptive lasso.
\newblock \emph{The international journal of biostatistics}, 13\penalty0 (2),
  2017.

\bibitem[van~der Laan(2023)]{van2023higher}
M.~van~der Laan.
\newblock Higher order spline highly adaptive lasso estimators of functional
  parameters: Pointwise asymptotic normality and uniform convergence rates.
\newblock \emph{arXiv preprint arXiv:2301.13354}, 2023.

\bibitem[van~der Laan and Rose(2011)]{van2011targeted}
M.~J. van~der Laan and S.~Rose.
\newblock \emph{Targeted learning: causal inference for observational and
  experimental data}.
\newblock Springer Science \& Business Media, 2011.

\bibitem[van~der Vaart and Wellner(1996)]{van1996weak}
A.~van~der Vaart and J.~Wellner.
\newblock \emph{Weak Convergence and Empirical Processes: With Applications to
  Statistics}.
\newblock Springer Science \& Business Media, 1996.

\bibitem[van~der Vaart(2000)]{van2000asymptotic}
A.~W. van~der Vaart.
\newblock \emph{Asymptotic statistics}, volume~3.
\newblock Cambridge university press, 2000.

\bibitem[van Keilegom and Veraverbeke(2001)]{van2001hazard}
I.~van Keilegom and N.~Veraverbeke.
\newblock Hazard rate estimation in nonparametric regression with censored
  data.
\newblock \emph{Annals of the Institute of Statistical Mathematics},
  53:\penalty0 730--745, 2001.

\bibitem[Walter and Blum(1984)]{walter1984simple}
G.~G. Walter and J.~R. Blum.
\newblock A simple solution to a nonparametric maximum likelihood estimation
  problem.
\newblock \emph{The Annals of Statistics}, pages 372--379, 1984.

\end{thebibliography}

\end{document}